\documentclass[11pt,reqno]{amsart}

 \usepackage{amsfonts,amsmath,amscd,amssymb,amsbsy,amsthm,amstext,amsopn,amsxtra}
 \usepackage{color,fullpage,mathrsfs}
 \usepackage{subfigure}
 \usepackage{verbatim} 
 \usepackage{stmaryrd}
 \usepackage{extarrows}
 \usepackage[all]{xy}
 \usepackage{bm}   
 \usepackage{hyperref}
 \usepackage{tikz}
 \usepackage{tikz-cd}
 \usepackage{enumitem}
 \usetikzlibrary{calc}

\usepackage{cancel}
\usepackage{tikz}

 \numberwithin{equation}{section}
 \hypersetup{colorlinks,linkcolor=[rgb]{0,0,1},citecolor=[rgb]{1,0,0}}

\newtheorem{theorem}{Theorem}[section]
\newtheorem{lemma}[theorem]{Lemma}
\newtheorem{proposition}[theorem]{Proposition}
\newtheorem{corollary}[theorem]{Corollary}
\newtheorem{conjecture}[theorem]{Conjecture}
\newtheorem{assumption}[theorem]{Assumption}
\newtheorem{property}[theorem]{Property}

\theoremstyle{definition}
\newtheorem{definition}[theorem]{Definition}
\newtheorem{example}[theorem]{Example}

\newtheorem{remark}[theorem]{Remark}
\newtheorem{remarks}[theorem]{Remarks}

\newcommand{\one}{\ensuremath{(\mathrm{i})}}
\newcommand{\two}{\ensuremath{(\mathrm{ii})}}
\newcommand{\three}{\ensuremath{(\mathrm{iii})}}

\newcommand{\kk}{\ensuremath{\Bbbk}} 
 
\newcommand{\QQ}{\ensuremath{\mathbb{Q}}}

\newcommand{\ZZ}{\ensuremath{\mathbb{Z}}}

 \newcommand{\modA}{\operatorname{mod}(\ensuremath{A})}
\newcommand{\chamber}{C}
\newcommand{\coh}{\operatorname{coh}}

\newcommand{\Dfin}{D^b_{\operatorname{fin}}}

\newcommand{\Dperf}{D^b_{\operatorname{perf}}}
\newcommand{\End}{\operatorname{End}}
\newcommand{\Ext}{\operatorname{Ext}}
\newcommand{\ghilb}{\ensuremath{\Gamma}\operatorname{-Hilb}}

\newcommand{\Hom}{\operatorname{Hom}}
\newcommand{\id}{\operatorname{id}} 

\newcommand{\Kfin}{K_{\operatorname{fin}}}

\newcommand{\Kperf}{K_{\operatorname{perf}}}
\newcommand{\Knum}{K^{\operatorname{num}}}
\newcommand{\Knumc}{K^{\operatorname{num}}_c}
\newcommand{\Knumperf}{K_{\operatorname{perf}}^{\operatorname{num}}}
\newcommand{\ltensor}{\overset{\mathbf{L}}{\otimes}}

\newcommand{\NS}{\ensuremath{N^1}}

\newcommand{\Pic}{\operatorname{Pic}}

\newcommand{\rank}{\operatorname{rk}}

\newcommand{\SL}{\operatorname{SL}}
\newcommand{\SO}{\operatorname{SO}}
\newcommand{\Spec}{\operatorname{Spec}}
\newcommand{\supp}{\operatorname{supp}}

\title{Gale duality and the linearisation map for noncommutative crepant resolutions}

\author{Alastair Craw} 
\address{Department of Mathematical Sciences, 
University of Bath, 
Claverton Down, 
Bath BA2 7AY, 
UK.}
\email{a.craw@bath.ac.uk}
\urladdr{http://people.bath.ac.uk/ac886/}

\begin{document}

\begin{abstract}
 We describe in geometric terms the map that is Gale dual to the linearisation map for quiver moduli spaces associated to
 noncommutative crepant resolutions in dimension three. This allows us to formulate Reid's recipe in this context in terms of a pair of integer-valued matrices, one of which appears to satisfy an attractive sign-coherence property. We provide some new evidence for a conjecture, known to hold in the toric case, and we deduce that a minimal generating set of relations between the determinants of the tautological bundles encodes the supports of the images of the vertex simples under the derived equivalence, and vice-versa.
\end{abstract}

\maketitle
 \tableofcontents
 
\section{Introduction}

 \subsection{The two key maps}
 The notion of a noncommutative crepant resolution (NCCR)~\cite{VdB04}, and more generally, that of a maximal modification algebra (MMA) \cite{IyamaWemyss14}, unifies the study of several natural classes of noncommutative threefold, including the skew group algebra of a finite subgroup of $\SL(3,\kk)$, the Jacobi algebra for a consistent dimer model on a real two-torus, and the endomorphism algebra of a tilting bundle on a local threefold flopping contraction~\cite{Wemyss18,Karmazyn17}. Every such algebra admits a presentation as a quiver algebra $A\cong \kk Q/I$, and a tilting bundle on a fine quiver moduli space $X=\mathcal{M}_\theta(A,v)$ determines a derived equivalence between $A$ and $X$. Here we consider two apparently quite different collections of geometric objects on $X$ arising from $A$ via the derived equivalence, and we uncover a strong link between the two that we phrase in terms of Gale duality between two abelian group homomorphisms. 

 On one hand, the indecomposable projective $A$-modules $P_i$ for vertices $i\in Q_0$ determine the tautological locally free sheaves $T_i$ on $X$. Our first map, known as the \emph{linearisation map} for the quiver moduli space $\mathcal{M}_\theta(A,v)$, is a map of finitely generated and free abelian groups induced by sending the class of each $P_i$ in the Grothendick group $K(A)$ to the numerical class of the line bundle $\det(T_i)$ on $X$. For a fixed algebra $A$ and an indivisible dimension vector $v$, this map depends on the GIT chamber $C$ containing the generic stability parameter $\theta$, so we write it as
 \[
 L_C\colon \Theta\longrightarrow \NS(X),
 \]
 where $\Theta$ is a corank-one lattice in $K(A)$ and $\NS(X)$ is the N\'{e}ron--Severi group of Cartier divisor classes on $X$ taken up to numerical equivalence. Note that $L_C(\theta)$ is the polarising ample line bundle on $\mathcal{M}_\theta(A,v)$ given by the GIT quotient construction.
 
 On the other hand, the vertex simple modules $S_i$ for $i\in Q_0$ determine objects $\Psi(S_i)$ in the bounded derived category of coherent sheaves with compact support on $X$; here, $\Psi$ is the tilting equivalence determined by the (dual to the) tautological bundle on $\mathcal{M}_\theta(A,v)$. The numerical Grothendieck group $\Knumc(X)$ for objects with compact support on $X$ admits a dimension filtration $\{0\}\subseteq F_0\subseteq F_1\subseteq F_2=\Knumc(X)$, and we demonstrate that the map Gale dual (see Section~\ref{sec:Galedual}) to the linearisation map $L_C$ is a map of the form
 \[
 G_C\colon F_2/F_0\longrightarrow F_2/F_1
 \]
 that is defined in terms of the support (with multiplicities) of the objects $\Psi(S_i)$ for $i\in Q_0$.
  
 This paper describes the relationship between the line bundles $\det(T_i)$ and the objects $\Psi(S_i)$ for $i\in Q_0$ via Gale duality between $L_C$ and $G_C$.  In the toric case, these two collections of geometric objects have been encoded before for a specific GIT chamber $C_+$ via rival interpretations of the combinatorial cookery known as Reid's recipe~\cite{Reid97, CHT21}. Here we reconcile these rival interpretations by demonstrating that they are equivalent. Our approach via $L_{C_+}$ or $G_{C_+}$ allows us to formulate Reid's recipe in the non-toric case, and it suggests that a matrix representing $G_{C_+}$ is `sign-coherent'; this surprising phenomenon would follow from a proof of the Cautis--Logvinenko Conjecture~\ref{conj:CLintro}. 
   
   \subsection{The main result and an illuminating example}
   We now present the results in more detail. 
   We assume for now that $A$ is an NCCR\footnote{A version of Theorem~\ref{thm:GCintro} holds in the more general situation where $A$ is an MMA, see Section~\ref{sec:MMA}.} satisfying a technical condition (see Assumption~\ref{ass:KrullSchmidt}) which implies in particular that the vertex simple modules exist.  In addition, we must assume that $v$ is indivisible for the functor $\Psi$ to be a derived equivalence (see Theorem~\ref{thm:BKR}).
   
 \begin{theorem}
 \label{thm:GCintro}
 Let $A$ be an NCCR of $R$ satisfying Assumption~\ref{ass:KrullSchmidt} and assume that $v$ is indivisible.  For any GIT chamber $C\subset \Theta\otimes_{\ZZ} \QQ$, then:
 \begin{enumerate}
     \item[\one] the abelian group $F_2/F_0$ is equal to the quotient of the free abelian group $\bigoplus_{i\in Q_0} \ZZ[\Psi(S_i)]$ by the subgroup generated by $\sum_{i\in Q_0} v_i [\Psi(S_i)]$; and
     \item[\two] the Gale dual to $L_{\chamber}$ is the $\ZZ$-linear map $G_{\chamber}\colon F_2/F_0 \rightarrow F_2/F_1$
 satisfying  
 \begin{equation}
     \label{eqn:GCintro}
 G_{\chamber}\big([\Psi(S_i)]\big) = \sum_{k\in \ZZ} (-1)^k \sum_{V} \ell_V\Big(\mathcal{H}^k\big(\Psi(S_i)\big)\Big)[\mathcal{O}_{V}]
 \end{equation}
  for $i\in Q_0$, where the second sum runs over all irreducible surfaces $V$ in the support of the cohomology sheaf $\mathcal{H}^k\big(\Psi(S_i))$, and where $\ell_V(-)$ computes the length of the stalk of the sheaf at the generic point of $V$.
 \end{enumerate}
\end{theorem}
  
 Even the doting parents of this result would admit that the baby is not very attractive, but its charming personality is revealed when $A$ is a toric algebra and $v=(1,\dots, 1)$. Under these assumptions, we break the symmetry in Theorem~\ref{thm:GCintro}\one\ by choosing a vertex $0\in Q_0$ so that $F_2/F_0$  is the free abelian group generated by the classes $[\Psi(S_i)]$ for $i\neq 0$. In addition, the choice of vertex $0$ allows us to define a special GIT chamber $C_+$ (see \eqref{eqn:C+} for the definition) for which the following statement was established by the author with Bocklandt and Quintero-V\'{e}lez~\cite[Theorem~1.4]{BCQ15}, motivated by the insight of Cautis--Logvinenko~\cite[Theorem~1.1]{CL09}.
 
 \begin{property}
 \label{property}
 For any stability condition $\theta\in C_+$ and for each vertex $i\neq 0$, the cohomology sheaves $\mathcal{H}^k(\Psi(S_i))$ on the moduli space $\mathcal{M}_\theta(A,v)$ are non-zero for a unique integer $k=k(i)\in \{-1,0\}$.
 \end{property}
 
 When Property~\ref{property} holds, then for any vertex $i\neq 0$, the first sum from \eqref{eqn:GCintro} collapses, and the support with multiplicities of the object $\Psi(S_i)$ can be read off directly from \eqref{eqn:GCintro}. 
  
 \begin{example}
 \label{exa:1315}
 Consider the action of type $\frac{1}{19}(1,3,15)$, so $\Gamma:=\ZZ/19\subset \SL(3,\kk)$ acts on $\mathbb{A}^3$ with generator $\text{diag}(\varepsilon, \varepsilon^3,\varepsilon^{15})$ for $\varepsilon$ a primitive 19th root of unity.  If we identify the vertex set of the McKay quiver $Q$ with $Q_0=\{0,1,\dots,18\}$, where vertex 0 corresponds to the trivial representation of $\Gamma$, then Reid's recipe ~\cite{Reid97,Craw05} labels every proper torus-invariant curve and surface in the crepant resolution $X=\mathcal{M}_\theta(A,v)\cong \Gamma\text{-Hilb}(\mathbb{A}^3)$ of $\mathbb{A}^3/\Gamma$ with a non-zero vertex of $Q$. 
\begin{figure}[!ht]
\begin{tikzpicture}[xscale=2,yscale=2]
\draw (0,0) -- (2,3.46) -- (4,0) -- (0,0); 
\draw (0,0) node[circle,draw,fill=black,minimum size=5pt,inner sep=1pt] {{\small .}};
\draw (0.737,0.183) node[circle,draw,fill=black,minimum size=5pt,inner sep=0pt] {{\small .}};
\draw (1.157,1.277) node[circle,draw,fill=black,minimum size=5pt,inner sep=0pt] {{\small .}};
\draw (1.578,2.37) node[circle,draw,fill=black,minimum size=5pt,inner sep=0pt] {{\small .}};
\draw (2,3.46) node[circle,draw,fill=black,minimum size=5pt,inner sep=0pt] {{\tiny .}};
\draw (1.473,0.367) node[circle,draw,fill=black,minimum size=5pt,inner sep=0pt] {{\tiny .}};
\draw (1.893,1.46) node[circle,draw,fill=black,minimum size=5pt,inner sep=0pt] {{\tiny .}};
\draw (2.314,2.553) node[circle,draw,fill=black,minimum size=5pt,inner sep=0pt] {{\tiny .}};
\draw (2.21,0.55) node[circle,draw,fill=black,minimum size=5pt,inner sep=0pt] {{\tiny .}};
\draw (2.631,1.644) node[circle,draw,fill=black,minimum size=5pt,inner sep=0pt] {{\tiny .}};
\draw (2.946,0.734) node[circle,draw,fill=black,minimum size=5pt,inner sep=0pt] {{\tiny .}};
\draw (4,0) node[circle,draw,fill=black,minimum size=5pt,inner sep=0pt] {{\tiny .}};
\draw (0,0) -- (2.946,0.734) -- (4,0); 
\draw (0,0) -- (1.157,1.277) -- (2.631,1.644) -- (4,0); 
\draw (0,0) -- (1.578,2.37) -- (2.314,2.553) -- (4,0); 
\draw (2,3.46) -- (0.737,0.183) -- (4,0); 
\draw (2.314,2.553) -- (1.473,0.367) -- (4,0); 
\draw (2.631,1.644) -- (2.21,0.55) -- (4,0); 
\draw (2,3.46) -- (2.946,0.734); 
\draw (1.578,2.37) -- (2.21,0.55);
\draw (1.157,1.277) -- (1.473,0.367);  
\draw (2,3.46) -- node [circle,draw,fill=white,inner sep=1pt] {\textcolor{red}{{\small 3}}} (2.314,2.553);
\draw (2.314,2.553) -- node [circle,draw,fill=white,inner sep=1pt] {\textcolor{red}{{\small 3}}} (2.631,1.644);
\draw (2.631,1.644) -- node [circle,draw,fill=white,inner sep=1pt] {\textcolor{red}{{\small 3}}} (2.946,0.734);
\draw (0,0) -- node [circle,draw,fill=white,inner sep=1pt] {\textcolor{red}{{\small 3}}} (0.737,0.183);
\draw (0.737,0.183) -- node [circle,draw,fill=white,inner sep=1pt] {\textcolor{red}{{\small 3}}} (1.473,0.367);
\draw (1.473,0.367) -- node [circle,draw,fill=white,inner sep=1pt] {\textcolor{red}{{\small 3}}} (2.21,0.55);
\draw (2.21,0.55) -- node [circle,draw,fill=white,inner sep=1pt] {\textcolor{red}{{\small 3}}} (2.946,0.734);
\draw (2.946,0.734) -- node [circle,draw,fill=white,inner sep=1pt] {\textcolor{red}{{\small 3}}} (4,0);
\draw (1.578,2.37) -- node [circle,draw,fill=white,inner sep=1pt] {\textcolor{red}{{\small 7}}} (1.893,1.46);
\draw (1.893,1.46) -- node [circle,draw,fill=white,inner sep=1pt] {\textcolor{red}{{\small 7}}} (2.21,0.55);
\draw (2.21,0.55) -- node [circle,draw,fill=white,inner sep=1pt] {\textcolor{red}{{\small 7}}} (4,0);
\draw (1.157,1.277) -- node [circle,draw,fill=white,inner sep=0.5pt] {\textcolor{red}{{\small 11}}} (1.473,0.367);
\draw (2,3.46) -- node [circle,draw,fill=white,inner sep=0.5pt] {\textcolor{red}{{\small 15}}} (1.578,2.37);
\draw (1.578,2.37) -- node [circle,draw,fill=white,inner sep=0.5pt] {\textcolor{red}{{\small 15}}} (1.157,1.277);
\draw (1.157,1.277) -- node [circle,draw,fill=white,inner sep=0.5pt] {\textcolor{red}{{\small 15}}} (0.737,0.183);
\draw (0.737,0.183) -- node [circle,draw,fill=white,inner sep=0.5pt] {\textcolor{red}{{\small 15}}} (4,0);
\draw (1.473,0.367) -- node [circle,draw,fill=white,inner sep=0.5pt] {\textcolor{red}{{\small 11}}} (4,0);
\draw (0,0) -- node [circle,draw,fill=white,inner sep=1pt] {\textcolor{red}{{\small 2}}} (1.157,1.277);
\draw (1.157,1.277) -- node [circle,draw,fill=white,inner sep=1pt] {\textcolor{red}{{\small 2}}} (1.893,1.46);
\draw (1.893,1.46) -- node [circle,draw,fill=white,inner sep=1pt] {\textcolor{red}{{\small 2}}} (2.631,1.644);
\draw (2.631,1.644) -- node [circle,draw,fill=white,inner sep=1pt] {\textcolor{red}{{\small 2}}} (4,0);
\draw (0,0) -- node [circle,draw,fill=white,inner sep=1pt] {\textcolor{red}{{\small 1}}} (1.578,2.37);
\draw (1.578,2.37) -- node [circle,draw,fill=white,inner sep=1pt] {\textcolor{red}{{\small 1}}} (2.314,2.553);
\draw (2.314,2.553) -- node [circle,draw,fill=white,inner sep=1pt] {\textcolor{red}{{\small 1}}} (4,0);
\draw (1.473,0.367) -- node [circle,draw,fill=white,inner sep=0.5pt] {\textcolor{red}{{\small 12}}} (1.893,1.46);
\draw (1.893,1.46) -- node [circle,draw,fill=white,inner sep=0.5pt] {\textcolor{red}{{\small 12}}} (2.314,2.553);
\draw (2.21,0.55) -- node [circle,draw,fill=white,inner sep=1pt] {\textcolor{red}{{\small 9}}} (2.631,1.644);
\draw (0.737,0.183) node[rectangle,draw,fill=white, inner sep=2pt] {\textcolor{blue}{{\small{18}}}};
\draw (1.473,0.367) node[rectangle,draw,fill=white, inner sep=2pt] {\textcolor{blue}{{\small{14}}}};
\draw (2.21,0.55) node[rectangle,draw,fill=white, inner sep=2pt] {\textcolor{blue}{{\small{10}}}};
\draw (2.946,0.734) node[rectangle,draw,fill=white, inner sep=2pt] {\textcolor{blue}{{\small{6}}}};
\draw (1.157,1.277) node[rectangle,draw,fill=white, inner sep=2pt] {\textcolor{blue}{{\small{17}}}};
\draw (1.578,2.37) node[rectangle,draw,fill=white, inner sep=2pt] {\textcolor{blue}{{\small{16}}}};
\draw (2.314,2.553) node[rectangle,draw,fill=white, inner sep=2pt] {\textcolor{blue}{{\small{4}}}};
\draw (2.631,1.644) node[rectangle,draw,fill=white, inner sep=2pt] {\textcolor{blue}{{\small{5}}}};
\draw (1.893,1.46) node[rectangle,draw,fill=white, inner sep=2pt] {\textcolor{blue}{{\small{8,13}}}};
\draw (2,3.62) node {{$e_1$}};
\draw (-0.2,0) node {{$e_3$}};
\draw (4.2,0) node {{$e_2$}};
\end{tikzpicture} 
\label{fig:juniorsimplex}
\caption{The junior simplex $\Delta$ defining $\ghilb(\mathbb{A}^3)$ for the action of type $\frac{1}{19}(1,3,15)$}
\end{figure}

We draw the height-one slice $\Delta$ of the toric fan $\Sigma$ of $X$ in Figure~\ref{fig:juniorsimplex} and label the interior lattice points and line segments by vertices of $Q$ according to the recipe. Once we make an arbitrary choice between $8$ and $13$ that mark one lattice point (we choose $8$), the geometric interpretation of Reid~\cite{Reid97} gives a $\ZZ$-basis for $\NS(X)$; in this case, we get
\begin{equation}
    \label{eqn:NSbasis}
\NS(X)\cong \ZZ\langle T_1, T_2, T_3, T_7, T_8, T_9, T_{11}, T_{12}, T_{15}\rangle.
\end{equation}
The labelling of the lattice points in Figure~\ref{fig:juniorsimplex} determines a minimal generating set of relations in $\NS(X)$ by \cite[Theorem~6.1]{Craw05}. For example, the points labelled $4$, $6$ and $\{8,13\}$ indicate that 
\[
T_4\cong T_1\otimes T_3, \quad T_6\cong T_3\otimes T_3; \quad \text{and}\quad T_8\otimes T_{13}\cong T_2\otimes T_7\otimes T_{12}
\]
hold respectively. If we list the $\ZZ$-basis of $\Theta$ as $[P_1],\dots, [P_{18}]$, then the matrix $L$ representing the linearisation map $L:=L_{C_+}$ with respect to this basis on $\Theta$ and that of $\NS(X)$ from \eqref{eqn:NSbasis} is 
\[
\setlength{\arraycolsep}{2pt}
   \renewcommand{\arraystretch}{0.8}
L=\mbox{\footnotesize$\left(\begin{array}{cccccccccccrrrcccc}
1 & 0 & 0 & 1 & 0 & 0 & 0 & 0 & 0 & 0 & 0 & 0 & 0  & 0 & 0 & 1 & 0 & 0\\
0 & 1 & 0 & 0 & 1 & 0 & 0 & 0 & 0 & 0 & 0 & 0 & 1 & 0 & 0 & 0 & 1 & 0\\
0 & 0 & 1 & 1 & 1 & 2 & 0 & 0 & 0 & 1 & 0 & 0 & 0  & 1 & 0 & 0 & 0 & 1\\
0 & 0 & 0 & 0 & 0 & 0 & 1 & 0 & 0 & 1 & 0 & 0 & 1  & 0 & 0 & 0 & 0 & 0\\
0 & 0 & 0 & 0 & 0 & 0 & 0 & 1 & 0 & 0 & 0 & 0 & -1 & 0 & 0 & 0 & 0 & 0\\
0 & 0 & 0 & 0 & 0 & 0 & 0 & 0 & 1 & 0 & 0 & 0 & 0 & 0 & 0 & 0 & 0 & 0\\
0 & 0 & 0 & 0 & 0 & 0 & 0 & 0 & 0 & 0 & 1 & 0 & 0  & 1 & 0 & 0 & 0 & 0\\
0 & 0 & 0 & 0 & 0 & 0 & 0 & 0 & 0 & 0 & 0 & 1 & 1  & 0 & 0 & 0 & 0 & 0\\
0 & 0 & 0 & 0 & 0 & 0 & 0 & 0 & 0 & 0 & 0 & 0 & 0  & 0 & 1 & 1 & 1 & 1
\end{array}\right).$}
\]
 A matrix defining the map $G:=G_{C_+}$ 
 is the transpose of a matrix $K$ representing the kernel of $L$. To compute this, list the minimal relations between the $T_i$ in ascending order of the vertices marking the corresponding interior lattice points, and create one column of $K$ for every such relation; the exponents of $T_1^{-1}\otimes T_3^{-1}\otimes T_4\cong \mathcal{O}_X$ give the first column. The transpose matrix 
\begin{equation}
    \label{eqn:Kt1315intro}
\setlength{\arraycolsep}{2pt}
   \renewcommand{\arraystretch}{0.8}
K^t=\mbox{\footnotesize$\left(\begin{array}{rrrrrrrrrrrrrrrrrr}
-1 & 0 & -1 & 1 &  0  &  0 & 0 & 0 & 0 & 0  & 0 & 0 & 0 &0 & 0  & 0 & 0 & 0 \\
0 & -1 & -1 & 0 & 1  &  0 & 0 & 0 & 0 & 0  & 0 & 0 & 0 & 0 & 0  & 0 & 0 & 0 \\
 0 & 0 & -2 & 0 &  0  &  1 & 0 & 0 & 0 & 0  & 0 & 0 & 0 & 0 & 0  & 0 & 0 & 0 \\
 0 & -1 & 0 & 0 &  0  &  0 & -1 &  1 & 0 & 0 & 0 & -1 & 1 & 0 & 0  & 0 & 0 & 0 \\
 0 &  0 & -1 & 0 & 0  &  0 & -1 & 0 & 0 &  1 & 0 & 0 & 0 & 0 & 0  & 0 & 0 & 0 \\ 
 0 &  0 & -1 &0 &  0  &  0 & 0 & 0 & 0 &  0 & -1 & 0 & 0 & 1 & 0  & 0 & 0 & 0\\ 
 -1 & 0 & 0 & 0 &  0  &  0 & 0 & 0 & 0 & 0 & 0 & 0 & 0 & 0 & -1 &  1 & 0 & 0 \\
 0 &  -1 & 0 & 0 &  0 &  0 & 0 & 0 &0 &  0 & 0 & 0 & 0 & 0 & -1 & 0 &  1 & 0 \\ 
 0 & 0 & -1 & 0 &  0  & 0 &  0 & 0 & 0 & 0 & 0 & 0 & 0 & 0 & -1 & 0 & 0 &  1   
\end{array}\right)$}
\end{equation}
 represents the map $G$ from Theorem~\ref{thm:GCintro}. Since Property~\ref{property} holds in the toric case, we obtain the following trichotomy: 
  \begin{enumerate}
      \item[(+)] examining the columns of $K^t$ that contain a positive entry shows that we have $k(i)=0$ for $i\in \{4, 5, 6, 8, 10, 13, 14, 16, 17, 18\}$. Theorem~\ref{thm:GCintro} implies that $\supp(\Psi(S_i))$ is the irreducible surface corresponding to the unique nonzero entry in the $i^{\text{th}}$ column of $K^t$; in each case, this corresponds to the lattice point in Figure~\ref{fig:juniorsimplex} marked with vertex $i$. 
      \item[(0)] the zero column of $K^t$ shows that $[\Psi(S_9)]\in \ker(G)$. Then $[\Psi(S_9)]\in F_1/F_0$, so $\Psi(S_9)$ is supported on curves. The ninth row of the matrix $L^t$ representing $\ker(G)$ shows that $\supp(\Psi(S_9))$ is the unique curve defined by the line segment in Figure~\ref{fig:juniorsimplex} marked by 9, as this curve $\ell$ satisfies $\deg(T_9\vert_\ell) = 1$ and $\deg(T_i\vert_\ell) = 0$ for $i\neq 9$;  and 
      \item[($-$)] the columns of $K^t$ containing a negative entry demonstrate that $k(i)=-1$ for each vertex $i\in \{1, 2, 3, 7, 11, 12, 15\}$. Theorem~\ref{thm:GCintro} implies that $\supp(\Psi(S_i))$ is the union of surfaces corresponding to the nonzero entries in the $i^{\text{th}}$ column of $K^t$, that is, to the lattice points in Figure~\ref{fig:juniorsimplex} that provide endpoints of two or more line segments marked with $i$. Note that $\mathcal{H}^{-1}(\Psi(S_3))$ has length two on the copy of $\mathbb{P}^2$ in $X$ indexed by the third row of $K^t$.
  \end{enumerate}
 \end{example}

The cases $(+)$ and $(-)$ from Example~\ref{exa:1315} illustrate that the support of the object $\Psi(S_i)$ can be read off directly from the matrix $K^t$, while in case $(0)$, the support of $\Psi(S_9)$ is computed via $L^t$. This can be carried out for any finite abelian subgroup $\Gamma\subset \SL(3,\kk)$.

 \subsection{Reconciling rival versions of Reid's recipe}
 The combinatorics determining the labelling in Figure~\ref{fig:juniorsimplex} is called \emph{Reid's recipe}~\cite{Reid97,Craw05, CHT21}, and it can be carried out whenever $A$ is a toric NCCR, $v=(1,\dots, 1)$ and we choose the chamber $C_+$. Write $L:=L_{C_+}$ for the linearisation map.
 
 In the special case when $A$ is the skew group algebra of a finite abelian subgroup $\Gamma\subset\SL(3,\kk)$, two quite different geometric interpretations of this recipe appear in the literature:
 \begin{enumerate}
     \item[(RR1)] 
     a $\ZZ$-basis of $\NS(X)$ comprising a subset of the tautological line bundles, together with a minimal set of relations between all of the tautological bundles~\cite{Reid97, Craw05}; and
     \item[(RR2)]  the calculation of the support of each object $\Psi(S_i)$ for each $i\neq 0$ in terms of the curves and surfaces marked by $i$ \cite{CL09, Logvinenko10} (see Definition~\ref{def:RR2} for a more precise statement).
 \end{enumerate}
  It is straightforward to see that interpretation (RR1) is encoded by the matrices $K$ and $L$ representing the nontrivial maps in the short exact sequence
 \begin{equation}
 \label{eqn:LCsequenceintro}
 \begin{tikzcd}
0 \ar[r] & \ker(L)\ar[r] & \Theta\ar[r,"L"]&  \NS(X) \ar[r]&  0.
\end{tikzcd}
  \end{equation}
 Theorem~\ref{thm:GCintro} shows that the Gale dual map $G:=G_{C_+}$ fits into the dual short exact sequence
  \begin{equation}
 \label{eqn:GCsequenceintro}
 \begin{tikzcd}
0 & F_2/F_1\ar[l] & F_2/F_0\ar[l,swap,"G"] &  F_1/F_0 \ar[l]&  0\ar[l].
\end{tikzcd}
  \end{equation}
  It can be shown (see Proposition~\ref{prop:RR2}) that the geometric interpretation of Reid's recipe in (RR2) is encoded by the sequence \eqref{eqn:GCsequenceintro}, or equivalently, by the transpose matrices $L^t$ and $K^t$. Thus, (see Theorem~\ref{thm:reconciliation}) the statements (RR1) and (RR2) are equivalent in this special case.
  
 More generally, we establish the following result in the toric case (for a more precise statement involving the entries of the matrix $K$, see Theorem~\ref{thm:CHT}). 
 
  \begin{theorem}
  \label{thm:rivalsintro}
 Let $A$ be a toric NCCR, set $v=(1,\dots, 1)$ and choose the GIT chamber $C_+$. The rival geometric interpretations of Reid's recipe from \emph{(RR1)} and \emph{(RR2)} are equivalent.
 \end{theorem}
  
  This result is nothing more than a comfort blanket
  when $A$ is the skew group algebra of a finite abelian subgroup $\Gamma\subset\SL(3,\kk)$, because (RR1) and (RR2) are known to hold in that case. However, Theorem~\ref{thm:rivalsintro} is new for a general toric NCCR. In fact, both (RR1) and (RR2) were formulated as conjectures in joint work with Heuberger and Tapia Amador (see Conjectures~\ref{conj:RR1}, \ref{conj:RR2}), and while both conjectures remain open, Theorem~\ref{thm:CHT} does at least show that the conjectures are equivalent. 

 \begin{remark}
 \emph{Derived Reid's recipe} for a finite abelian subgroup of $\SL(3,\kk)$ from \cite{CCL17}, where each $\Psi(S_i)$ is computed explicitly, is strictly stronger than the numerical data encoded by $K^t$ and $L^t$. 
  \end{remark}
 
 \subsection{The general case}
 Example~\ref{exa:1315} illustrates not only the calculation of the matrix $K^t$, but also the importance of Property~\ref{property}, without which we are unable to draw conclusions about the support of $\Psi(S_i)$ directly from $G_{C_+}$. We anticipate that  Property~\ref{property} is not a purely toric phenomenon, and we also expect there is a link between the integers $k(i)$ from Property~\ref{property} and the walls of $C_+$. 
 
 \begin{conjecture}[The Cautis--Logvinenko conjecture]
 \label{conj:CLintro}
 Let $A\cong \kk Q/I$ be an NCCR of the ring $R$ satisfying Assumption~\ref{ass:KrullSchmidt}. Assume $v_0=1$ for some $0\in Q_0$, and fix the normalisation of the tautological bundle $T$ on $X=\mathcal{M}_\theta(A,v)$ for $\theta\in C_+$ so that $T_0\cong \mathcal{O}_X$. Then \begin{enumerate}
     \item[\one] Property~\ref{property} holds, that is, for each vertex $i\neq 0$, there is a unique $k=k(i)\in \{-1,0\}$ such that $\mathcal{H}^{k}(\Psi(S_i))\neq 0$; and
     \item[\two] $k(i)=0$ if and only if the locus $\{\theta\in \Theta \mid \theta_i=0\}$ is a supporting hyperplane of $C_+$. 
 \end{enumerate}
  \end{conjecture}
 
 \begin{remarks}
 \begin{enumerate}
 \item Both statements in this conjecture were formulated originally for the skew group algebra of a finite subgroup of $\SL(3,\kk)$: see~\cite[Conjecture~1.2]{CL09} for part \one, noting that the assumption $i\neq 0$ is essential~\cite{CL14}; and see \cite[Conjecture~5]{BCQcorrection} for part \two. 
 \item A proof of part \two\ in the toric case appears in \cite[Proposition~1.3]{BCQ15}. We present supporting evidence for part \two\ in some non-toric cases, see Examples~\ref{exa:HomMMPcont} and \ref{exa:BentoBox}. The link between Reid's recipe and the walls defining the GIT chamber $C_+$ was developed further in \cite{Wormleighton20}. 
 \end{enumerate}
 \end{remarks}
 
 Whenever Conjecture~\ref{conj:CLintro} holds, it follows (see Corollary~\ref{cor:GC}) that the matrix $K^t$ representing the linear map $G_{C_+}$ satisfies the following sign-condition:
 
 \begin{definition}
 \label{def:SC}
 A matrix is \emph{sign-coherent} if for each column, the nonzero entries are either all positive or all negative. 
 \end{definition}
 
 This rather striking feature is illustrated by the matrix $K^t$ from \eqref{eqn:Kt1315intro} above: the rows containing the nonzero entries in the $i^{th}$ column of $K^t$ index the irreducible divisors (with multiplicities) in the support of the object $\Psi(S_i)$ for $i\neq 0$.

 \begin{remark}When Property~\ref{property} holds, Theorem~\ref{thm:GCintro} allows us to draw geometric conclusions according to the three possibilities for the sign of the entries in the $i^{\text{th}}$ column of the sign-coherent matrix $K^t$ associated to the Gale dual map $G_{C_+}$: 
 \begin{enumerate}
     \item[(+)] if the nonzero entries are positive, then $\Psi(S_i)$ is a sheaf whose support contains surfaces;
     \item[(0)] if the column is the zero vector, then $\Psi(S_i)$ is a sheaf supported on curves; or
     \item[($-$)] if the nonzero entries are negative, then  $\Psi(S_i)[1]$ is a sheaf whose support contains surfaces.
 \end{enumerate}
 This trichotomy (for vertices $i\neq 0$) has long been a key feature of Reid's recipe in the toric case, see \cite[Theorem~1.2]{Logvinenko10} or \cite[Theorem~1.4]{BCQ15}.   
   \end{remark}
   
 Going further, the proof of Theorem~\ref{thm:rivalsintro} suggests that one should view Reid's recipe not as the 
 decoration of curves and surfaces by vertices of $Q$ as in Figure~\ref{fig:juniorsimplex}, but rather, as the 
 calculation of the matrices $K$ and $L$ defining the sequence \eqref{eqn:LCsequenceintro} that encodes the geometric information expressed in (RR1), or equivalently, the matrices $K^t$ and $L^t$ defining \eqref{eqn:GCsequenceintro} that encodes (RR2). We therefore take the obvious step in \emph{defining} Reid's recipe this way (see Definition~\ref{def:RR}) for any NCCR  $A$ satisfying Property~\ref{property}. A similar definition can be formulated for an MMA.

 We illustrate this change of viewpoint by calculating $L$ and $K^t$ for the example of a (non-abelian) dihedral subgroup of $\SL(3,\kk)$ that was studied in recent work of Nolla de Celis~\cite{NolladeCelis21}. Our summary in Example~\ref{exa:BentoBox} adds nothing of substance to the lovely calculations from \emph{ibid.}\  other than to provide some supporting evidence for Conjecture~\ref{conj:CLintro} in terms of the sign-coherent matrix $K^t$ from \eqref{eqn:Ktbento}.

\subsection*{Acknowledgements}
This project was inspired by the beautiful bento box calculation by \'{A}lvaro Nolla de Celis~\cite{NolladeCelis21}, and I am grateful to him for sharing an early version of his paper. I wrote Section~\ref{sec:numGrothGroups} as an appendix to a joint project with Arend Bayer and Ziyu Zhang, and I thank them both for allowing me to include those results here. In addition,  I'm grateful to Gwyn Bellamy and Matthew Pressland for useful conversations. Finally, I thank the anonymous referee for many helpful comments.

\section{Duality for numerical Grothendieck groups}
\label{sec:numGrothGroups}
 Throughout this section, $X$ denotes a separated scheme of finite type over an algebraically closed field $\kk$. We impose additional assumptions in Proposition~\ref{prop:Fmperp} and Corollary~\ref{cor:NS}.

\subsection{Numerical Grothendieck group for compact support}
Let $\Dperf(X)$ denote the bounded derived category of perfect complexes on $X$.  Let $D^{b}_{c}(X)$ denote
the bounded derived category of coherent sheaves on $X$ with proper support, where the support of a complex $E$ is  the union of the supports of its cohomology sheaves $\mathcal{H}^k(E)$ for $k\in \ZZ$. The Euler form  
\[
\chi\colon \Kperf(X)\times K_c(X)\to \ZZ
\]
between the Grothendieck groups of these categories is given by 
\begin{equation}
\label{eqn:euler}
\chi(E,F)= \sum_{i\in \ZZ} (-1)^i\dim_\kk \Hom(E,F[i]).
\end{equation}
 Following \cite{BCZ17}, the \emph{numerical Grothendieck group for compact support}, denoted $\Knumc(X)$, is
 defined to be
 the quotient of $K_c(X)$ by the radical of the form $\chi$. Each object $E\in D^{b}_{c}(X)$ defines a class $[E]\in \Knumc(X)$. When $\kk$ has characteristic zero, \cite[Lemma~5.1.1]{BCZ17} shows that $\Knumc(X)$ has finite rank when $X$ is normal and quasi-projective.
 
 For $m\geq 0$, define $F_m \text{K}_c(X)$ to be the subgroup of $K_c(X)$ generated by sheaves whose support is proper and of dimension at most $m$. After passing to the quotient by the radical of the Euler form, we obtain a subgroup $F_m\subseteq \Knumc(X)$, giving the \emph{dimension filtration}
  \[
 \{0\}=F_{-1}\subseteq F_0\subseteq F_1\subseteq \cdots \subseteq F_n = \Knumc(X),
 \]
 where $n$ is the maximal dimension of a proper subscheme of $X$. 
 
\begin{proposition}
\label{prop:Knumclass}
 For $E\in D^{b}_{c}(X)$, let $m$ denote the maximal dimension of an irreducible component of $\supp(E)$. The class of $E$ in $\Knumc(X)$ satisfies
 \[
 [E] =\sum_{k\in \ZZ} (-1)^k \sum_{V} \ell_V\big(\mathcal{H}^k(E)\big) [\mathcal{O}_V] \mod F_{m-1},
 \]
 where the second sum is taken over all irreducible components $V$ of $\supp(\mathcal{H}^k(E))$, and $\ell_V(-)$ denotes the length of the stalk of the sheaf at the generic point of $V$. 
\end{proposition}
\begin{proof}
 The class of $E$ in $\Knumc(X)$ satisfies $[E]=\sum_{k\in \ZZ} (-1)^k [\mathcal{H}^k(E)]$, so it suffices to prove for a coherent sheaf $E$ on $X$ with proper support that  \[
 [E] = \sum_{V} \ell_V(E) [\mathcal{O}_V] \mod F_{m-1},
 \]
 where the sum is taken over all irreducible components $V$ of $\supp(E)$. This follows by combining the result of \cite[Expose X, Proposition~1.1.2]{SGA6} with the fact that $\supp(E)$ is proper if and only if each irreducible component of $\supp(E)$ is proper.
\end{proof}
 
 \begin{corollary}
 \label{cor:GensFm}
 For $0\leq m\leq n$, the subgroup $F_m\subseteq \Knumc(X)$ is generated by the classes $[\mathcal{O}_V]$ defined by integral, proper subschemes $V\subseteq X$ of dimension at most $m$. 
 \end{corollary}
 
 \subsection{Numerical Grothendieck group for perfect complexes}
  Define Cartier divisors $D, D^\prime$ on $X$ to be numerically equivalent if $D\cdot C=D^\prime\cdot C$ for every proper curve $C$ in $X$. Following Lazarsfeld~\cite[Definition~1.1.15]{Lazarsfeld}, define the N\'{e}ron--Severi group $\NS(X)$ to be the group of numerical equivalence classes of Cartier divisors of $X$. 
The theorem of the base gives that $\NS(X)$ is a finitely generated and free abelian group.

  Define the \emph{numerical Grothendieck group for perfect complexes} on $X$, denoted $\Knumperf(X)$, to be the quotient of $\Kperf(X)$ by the radical of the Euler form $\chi$. Write 
 \[
 \chi\colon \Knumperf(X)\times \Knumc(X)\longrightarrow \ZZ
  \]
 for the induced perfect pairing. Each object $P\in \Dperf(X)$ defines a class $[P]\in \Knumperf(X)$. For $m\geq 0$, set
 \[
 F_m^\perp:=\big\{[P]\in \Knumperf(X) \mid \chi(P,E) = 0 \text{ for all }E\in F_m\big\}.
 \]

 Assume from now on that $X$ is connected. The group homomorphism $\rank\colon \Kperf(X)\to \ZZ$ that sends the class of an object $P\in \Dperf(X)$ to its rank is well-defined \cite[0BDJ]{stacks-project}.  In addition, consider the group homomorphism $\det\colon \Kperf(X)\to \NS(X)$ sending the class of an object $P\in \Dperf(X)$ to the numerical divisor class $[\det(P)]$ of its determinant line bundle \cite{KnudsenMumford}. 
  
 \begin{proposition}
 \label{prop:Fmperp}
 Let $X$ be connected. 
 \begin{enumerate}
     \item[\one] The group homomorphisms $\rank$ and $\det$ factor through the numerical Grothendieck groups, inducing group homomorphisms 
     \[
     \rank\colon \Knumperf(X)\to \ZZ\quad \text{and}
     \quad \det\colon \Knumperf(X)\to \NS(X).
     \]
     \item[\two] The subgroups $F^1:=\ker(\rank)$ and $F^2:= \ker(\det)\cap \ker(\rank)$ of $\Knumperf(X)$ satisfy 
     \[
     F^1=F_0^\perp\quad\text{ and }\quad F^2=F_1^\perp.
     \]
     \item[\three] The N\'{e}ron--Severi group of $X$ fits into a short exact sequence of abelian groups
      \[
\begin{tikzcd}
0 \ar[r] & F^2\ar[r] & F^1\ar[r,"\det"]&  \NS(X) \ar[r]&  0.
\end{tikzcd}
\]
 \end{enumerate} 
 \end{proposition}
 \begin{proof}
 Let $P\in \Dperf(X)$ satisfy $\chi(P,E)=0$ for all $E\in D^{b}_{c}(X)$. By combining the classes of the even and odd cohomology sheaves, we see that the class of $P$ in $\Kperf(X)$ equals can be written as the difference
 $[P_+]-[P_-]$, where $P_\pm$ are locally free sheaves of finite rank on $X$. For any closed point $x\in X$, we have 
 \[
 \rank(P_\pm) = \rank(P_\pm^\vee) = \dim_\kk H^0(P_\pm^\vee\vert_x) = \chi(P_\pm^\vee\vert_x) = \chi(P_\pm,\mathcal{O}_x).
 \]
 Both $\rank$ and $\chi$ are additive on exact sequences, so 
 \begin{equation}
\label{eqn:rankP}   
 \rank(P)=\rank(P_+)-\rank(P_-) = \chi(P_+,\mathcal{O}_x) - \chi(P_-,\mathcal{O}_x) = \chi(P,\mathcal{O}_x).
\end{equation}
This is zero by our assumption on $P$, so $\rank\colon \Knumperf(X)\to \ZZ$ is well-defined. Similarly, for a proper, irreducible curve $C\subseteq X$ and a closed point $x\in X$, we have 
\[
\deg\big(P_\pm^\vee\vert_C\big) = \chi(P_\pm^\vee\vert_C) - \rank(P^\vee_\pm) \chi(\mathcal{O}_C) = \chi(P_\pm, \mathcal{O}_C) - \chi(P_\pm,\mathcal{O}_x)\cdot \chi(\mathcal{O}_C).
\]
 Since $\deg(\det(P^\vee)\vert_C) = \deg(P^\vee\vert_C)$, combining this with additivity of $\deg$ and $\chi$ gives
 \begin{equation}
\label{eqn:degdetP}     
\deg\!\big(\det(P)^{-1}\vert_C\big) =  \deg\big(P_+^\vee\vert_C\big) - \deg\big(P_-^\vee\vert_C\big) = \chi(P,\mathcal{O}_C)-\chi(P,\mathcal{O}_x)\cdot \chi(\mathcal{O}_C)= 0,
 \end{equation}
 where the final equality again follows from our assumption on $P$. This holds for every proper, irreducible curve $C\subseteq X$, so $\det(P)^{-1}$ is numerically trivial and hence so is $\det(P)$. This determines the group homomorphism $\det\colon \Knumperf(X)\to \NS(X)$.

 For \two, let $P\in \Dperf(X)$. Equation \eqref{eqn:rankP} gives  $\rank(P)=\chi(P,\mathcal{O}_x)$ for any closed point $x\in X$, so $\ker(\rank)= F_0^\perp$ by Corollary~\ref{cor:GensFm}. For the second equality, Corollary~\ref{cor:GensFm} also gives
 \[
 [P]\in F_1^\perp \iff \chi(P,\mathcal{O}_C)= \chi(P,\mathcal{O}_x)= 0 
 \]
 for every proper irreducible curve $C\subseteq X$ and closed point $x\in X$. Then $\chi(P,\mathcal{O}_x)= 0$ is equivalent to $\rank(P)=0$, while $\chi(P,\mathcal{O}_C)=0$ is equivalent by \eqref{eqn:degdetP} to $[\det(P)^{-1}]=[\mathcal{O}_X]$ in $\NS(X)$. Combining both conditions shows that $[P]\in F_1^\perp$ if and only if $[P]\in F^2$.
 
 For \three, the restriction of $\det\colon \Knumperf(X)\to \NS(X)$ to $F^1=\ker(\rank)$ is surjective, because any line bundle $L$ representing a given numerical class $[L]\in \NS(X)$ satisfies $\det([L]-[\mathcal{O}_X]) = [L]$. The kernel of the resulting homomorphism $\det\colon F^1\to \NS(X)$ equals $F^2$ by construction. 
 \end{proof}
 
\begin{remark}
 We write the subgroups from Proposition~\ref{prop:Fmperp}\two\ as $F^2\subseteq F^1\subseteq \Knumperf(X)$ because $F^1$ is induced from the $\gamma$-filtration on the $\lambda$-ring $\Kperf(X)$ \cite[V, \S 3, Remark~1]{FultonLang85}. Indeed, Proposition~\ref{prop:Fmperp}\three\ is a numerical version of this well-known construction of $\Pic(X)$ in terms of the $\gamma$-filtration on $\Kperf(X)$. 
\end{remark}
 
 \begin{corollary}
 \label{cor:NS}
 Let $X$ be connected and suppose in addition that $X$ is proper over an affine base. The perfect pairing
 \begin{equation}
\label{eqn:euler2}
 \chi\colon \frac{F^1}{F^2}\times \frac{F_1}{F_0}\longrightarrow \ZZ
 \end{equation}
  induced by the Euler form coincides with the classical intersection pairing multiplied by $-1$.
 \end{corollary}
 \begin{proof}
 Write $f\colon X\to B$ for the morphism, where $B$ is affine. Every proper curve is contracted to a point by $f$, and hence every such curve is contained in a unique fibre of $f$. Our notion of numerical equivalence therefore coincides with that of Kleiman~\cite[IV, \S 4]{KleimanAnnals}, in which case Proposition~3 from \emph{ibid.}\  implies that the free abelian group $N^1(X)$ has finite rank. The same section of \emph{ibid.}\ shows that the intersection product between Cartier divisors on $X$ and irreducible proper curves in $X$ given by $D\cdot C = \deg \mathcal{O}_X(D)\vert_C$ induces a perfect pairing $N^1(X)\times N_1(X)\to \ZZ$, where $N_1(X)$ is the abelian group generated by irreducible, proper curves in $X$ modulo the relation that $\sum_i \lambda_i C_i, \sum_j \mu_j C_j$ are numerically equivalent if $\sum_i \lambda_i \deg L\vert_{C_i} = \sum_j \mu_j\deg L\vert_{C_j}$ for every $L\in \Pic(X)$.  
 
 Proposition~\ref{prop:Fmperp}\two\ implies that the Euler form induces the perfect pairing \eqref{eqn:euler2}. To compare this to the intersection pairing, first apply the isomorphism $\alpha\colon F^1/F^2\to \NS(X)$ from Proposition~\ref{prop:Fmperp}\three. The inverse map satisfies $\alpha^{-1}([L]) = [L]-[\mathcal{O}_X]\mod F^2$. Since
 \[
 \chi(\alpha^{-1}([L]),\mathcal{O}_C)  = \chi(L,\mathcal{O}_C)-\chi(\mathcal{O}_X,\mathcal{O}_C)  =  \deg L^{-1}\vert_C = -\deg L\vert_C,
 \]
 the perfect pairing from \eqref{eqn:euler2} is $-1$ times the intersection pairing $D\cdot C = \deg \mathcal{O}_X(D)\vert_C$. 
 \end{proof}

\section{Noncommutative crepant resolutions and quiver moduli}
We now introduce the algebras $A$ and the varieties $X$ that are of primary interest.

\subsection{Noncommutative crepant resolutions}
 Let $R$ be a finitely generated $\kk$-algebra that is normal, Gorenstein and of Krull dimension three. A $\kk$-algebra $A$ of finite global dimension is a \emph{noncommutative crepant resolution} (NCCR) \emph{of $R$} if there is a reflexive $R$-module $M$ such that $A:=\End_R(M)$ is Cohen--Macaulay as an $R$-module~\cite{VdB04}. Choose once and for all a decomposition $M \cong \bigoplus_{0\leq i\leq r} M_i$ into finitely many indecomposable, reflexive $R$-modules.
 
 \begin{lemma}
 \label{lem:quiverpresentation}
 There is a finite, connected quiver $Q$ and an ideal $I\subset \kk Q$ such that $A\cong \kk Q/I$.
  \end{lemma}
 \begin{proof}
 The chosen decomposition of $M:=\bigoplus_{0\leq i\leq r}M_i$ gives a complete set of orthogonal idempotents $e_i=\id\in \End(M_i)$ of $A$ such that $1=\sum_{0\leq i\leq r} e_i$. To construct the quiver $Q$, first define the vertex set $Q_0=\{0,1,\dots,r\}$. Next introduce a set of loops at each vertex corresponding to a finite set of $\kk$-algebra generators of $R$, and 
 for $0\leq i,j\leq r$, introduce arrows from $i$ to $j$ corresponding to a finite generating set for $\Hom(M_i,M_j)$ as an $R$-module. This determines a surjective $\kk$-algebra homomorphism $\kk Q\to \End(M)$ with kernel $I$. 
 \end{proof}
 
 From now on, we impose the following additional assumption on the presentation of $A$ from Lemma~\ref{lem:quiverpresentation} that appeared originally as \cite[Assumption~5.1]{CIK18}.
  
\begin{assumption}
 \label{ass:KrullSchmidt}
 \begin{enumerate}
 \item[\one] All indecomposable projective $A$-modules occur as summands of $A$, and a given presentation $A\cong\kk Q/I$ corresponds to a unique decomposition of $M= \bigoplus_{i \in Q_0} M_i$ into non-isomorphic, indecomposable reflexive modules; and
 \item[\two] the ideal $I\subset \kk Q$ is generated by linear combinations of paths of length at least one.
 \end{enumerate}
\end{assumption}
 
 \begin{remarks}
  \label{rems:idealidempotent}
 \begin{enumerate}
 \item If the module category of $R$ has the Krull-Schmidt property, then any NCCR is Morita equivalent to an NCCR satisfying Assumption \ref{ass:KrullSchmidt}\one.
\item Assumption~\ref{ass:KrullSchmidt}\two\ ensures that each vertex $i\in Q_0$ gives rise to a vertex simple $A$-module $S_i=\kk e_i$, and that distinct summands of $M$ are non-isomorphic (as an isomorphism $M_i\cong M_j$ forces the relations $aa^\prime-e_i, a^\prime a-e_j\in I$ for some $a, a^\prime \in Q_1$).
\end{enumerate}
\end{remarks}
    
  Assumption~\ref{ass:KrullSchmidt} is well known to hold for the following classes of examples.
 
 \begin{example}
 \label{exa:ghilb}
     For any finite subgroup $\Gamma\subset \SL(3,\kk)$, the skew group algebra $A=\kk[x,y,z]\star\Gamma$ is an NCCR of the $\Gamma$-invariant subalgebra $R=\kk[x,y,z]^\Gamma$. For each  irreducible representation $\rho$ of $\Gamma$, consider the $R$-module $M_\rho:=(\kk[x,y,z]\otimes_\kk \rho^*)^\Gamma$ where $\rho^*$ is dual to $\rho$. We have 
     \[
     A\cong\End_R\bigg(\bigoplus_{\rho} M_\rho\otimes_\kk \rho\bigg),
     \]
     so $A$ does not satisfy Assumption~\ref{ass:KrullSchmidt} if  $\dim_\kk \rho>1$ for some $\rho$. However, $A$ is Morita equivalent to  the algebra $\End_R(\bigoplus_{\rho} M_\rho)$ which does satisfy Assumption~\ref{ass:KrullSchmidt}. 
     \end{example}
 
 \begin{example}
 \label{exa:dimers}
 Let $A$ denote a CY-3 toric order in the sense of Bocklandt~\cite{Bocklandt12}. Then \emph{ibid}.\ shows that every such algebra $A$ arises as the path algebra modulo relations of a consistent dimer model on a real 2-torus. In this situation, Broomhead~\cite{Broomhead12} proves that $A$ is an NCCR of its centre $R$ which is a Gorenstein normal semigroup algebra of Krull dimension three. Conversely, every such algebra $R$ admits an NCCR that arises in this way by \cite{Broomhead12,IshiiUeda15}.
   \end{example}

 Let $D^b(A)$ (resp.\ $\Dfin(A)$) denote the bounded derived category of $A$-modules (resp. finite-dimensional $A$-modules), and write $K(A)$ (resp.\ $\Kfin(A)$) for the Grothendieck group of this category. Since $A$ has finite global dimension, the Euler form $\chi_A\colon K(A)\times \Kfin(A)\to \ZZ$ is the bilinear form satisfying
 \[
 \chi_A(E,F) = \sum_{i\in \ZZ} (-1)^i \dim_\kk \Ext^i_A(E,F)
 \]
 for $E\in K(A)$ and $F\in \Kfin(A)$. The \emph{numerical Grothendieck group} for $A$ is defined to be  
 \[
 \Knum(A):= \Kfin(A)/K(A)^\perp.
 \]
 Each vertex $i\in Q_0$ determines an indecomposable projective $A$-module $P_i:= Ae_i$. Since $A$ has finite global dimension, the classes $[P_i]\in K(A)$ for $i\in Q_0$ generate $K(A)$ over $\ZZ$. The following stronger statement is immediate from \cite[Lemma~7.1.1]{BCZ17}.
  
 \begin{lemma}
 \label{lem:NumKgroupNCCR}
 We have $K(A)\cong \bigoplus_{i\in Q_0} \ZZ[P_i]$, $\Knum(A)\cong \bigoplus_{i\in Q_0} \ZZ[S_i]$, and $\chi_A$ descends to a perfect pairing $\chi_A\colon K(A)\times \Knum(A)\to \ZZ$.
 \end{lemma}

 \subsection{Quiver moduli spaces}
 Define the vector $v = \sum_i v_i[S_i]\in \Knum(A)$ whose components satisfy $v_i=\rank_R(M_i)$ for $0\leq i\leq r$. After replacing $A$ by a Morita equivalent algebra if necessary, we may assume that $v$ is indivisible~\cite[Section 6.3]{VdB04}. An $A$-module $N$ of dimension vector $v$ satisfies $v_i=\dim_\kk e_iN$ for all $i\in Q_0$. Consider the lattice
    \[
   \Theta:=\big\{\theta\in K(A)=\Hom\big(\Knum(A),\ZZ\big) \mid \chi_A(\theta,v)=0\big\}.
   \]
 Write $\theta(v):=\chi_A(\theta,v)$. Lemma~\ref{lem:NumKgroupNCCR} implies that each $\theta\in \Theta$ is of the form $\theta=\sum_{i\in Q_0} \theta_i[P_i]$ with $\theta_i\in \ZZ$ for all $i\in Q_0$. Consider also the rational vector space $\Theta_{\QQ}:=\Theta\otimes_{\ZZ}\QQ$. 
 
 For any $\theta\in \Theta_{\QQ}$, an $A$-module $N$ of dimension vector $v$ is said to be $\theta$-semistable (resp.\ $\theta$-stable) if  $\theta(N'):= \theta(\dim_\kk N')\geq 0$ (resp.\ $\theta(N')>0$) for every $A$-submodule $0\subsetneq N'\subsetneq N$.  The presentation of $A$ as a quiver algebra from Lemma~\ref{lem:quiverpresentation} allows us to apply the work of King~\cite{King94} to construct the coarse moduli space $\mathcal{M}_\theta(A,v)$ of $\theta$-semistable $A$-modules of dimension vector $v$ (upto S-equivalence) as a GIT quotient. We say $\theta$ is \emph{generic with respect to $v$}, or simply \emph{generic} if the choice of $v$ is clear, if every $\theta$-semistable $A$-module of dimension vector $v$ is $\theta$-stable. As is standard in GIT, the set of generic parameters in $\Theta_{\QQ}$ decomposes into a union of GIT chambers, where $\theta, \theta'\in \Theta_{\QQ}$ lie in the same chamber if and only if the notions of $\theta$-stability and $\theta'$-stability coincide. 
 
 For any chamber $C\subseteq \Theta_{\QQ}$ and for any $\theta\in C$, King~\cite[Proposition~5.3]{King94} established that $\mathcal{M}_\theta(A,v)$ is the fine moduli space of $\theta$-stable $A$-modules of dimension vector $v$ (up to isomorphism). In this case, the moduli space $\mathcal{M}_\theta(A,v)$ carries a tautological vector bundle
 \[
 T=\bigoplus_{i\in Q_0}T_i
 \]
 together with a tautological $\kk$-algebra homomorphism $A\to \End(T)$, such that the fibre of $T$ over any closed point $[N]\in \mathcal{M}_\theta(A,v)$ is the corresponding $\theta$-stable $A$-module $N$.  As King notes, the tautological bundle $T$ is well-defined only up to tensor product by a line bundle.
  
  The following result of Van den Bergh~\cite[Theorem~6.3.1, Remark~6.6.1]{VdB04} builds on the original work by Bridgeland--King--Reid~\cite{BKR01}. For a CY-3 toric order $A$, an independent proof of this result is due to  Ishii--Ueda~\cite{IshiiUeda08,IshiiUeda15}.
  Note that a simplified proof of part \one\ appears in \cite[Proposition~A.3]{CIK18}. 

\begin{theorem}
\label{thm:BKR}
Let $A$ be an NCCR of $R$ satisfying Assumption~\ref{ass:KrullSchmidt} and assume that $v$ is indivisible. Then for any chamber $C\subseteq \Theta_{\QQ}$ and for any $\theta\in C$, we have that:
\begin{enumerate}
\item[\one] the fine moduli space $X:=\mathcal{M}_\theta(A,v)$ is connected;
\item[\two] there is a projective crepant resolution $f_\theta\colon X\to \Spec R$; and
\item[\three] the bundle $T^\vee:=\mathcal{H}om_{\mathcal{O}_X}(T,\mathcal{O}_X)$ dual to the tautological bundle $T$
on $\mathcal{M}_\theta(A,v)$ is a tilting bundle, giving an equivalence 
 \begin{equation}
 \label{eqn:derivedequivPhi}
 \Phi(-):= \mathbf{R}\!\Hom_{\mathcal{O}_X}\!(T^\vee,-)\colon D^b(X)\longrightarrow D^b(A)
 \end{equation}
 of triangulated categories determined by $T^\vee$, with quasi-inverse
 \begin{equation}
 \label{eqn:derivedequivalenceadjoint}
 \Psi(-):= T^\vee\otimes_{A}- \colon D^b(A)\longrightarrow D^b(X).
 \end{equation} 
\end{enumerate}
\end{theorem}

\begin{remarks}
\label{rems:phipsi}
\begin{enumerate}
    \item Note that $\Phi_C(E) = \mathbf{R}\!\Hom_Y(T^\vee,E)$ is a right module over $\End(T^\vee)$ and hence a left module over $\End(T^\vee)^{\text{op}}\cong \End(T) \cong (\Phi\circ \Psi)(A)\cong A$.
    \item For each $i\in Q_0$, we have $\Psi(P_i) = T^\vee\otimes_A A e_i = T_i^\vee$. 
 \item The derived equivalences $\Phi$ and $\Psi$ induce $\ZZ$-linear isomorphisms on Grothendieck groups $\varphi\colon K(X)\to K(A)$ and $\psi:=\varphi^{-1}$ respectively.  Combining (2) above with Lemma~\ref{lem:NumKgroupNCCR} shows that $K(X)\cong \bigoplus_{i\in Q_0} \ZZ[T_i^\vee]$.
 \end{enumerate}
\end{remarks}
 
 Let $C\subset \Theta_{\QQ}$ be a chamber and write $X:= \mathcal{M}_\theta(A,v)$ for $\theta\in C$.  The \emph{linearisation map} for the chamber $C$ is the $\ZZ$-linear map $L_C\colon \Theta \rightarrow \NS(X)$ defined by
  \begin{equation}
 \label{eqn:LC}
 L_{\chamber}(\eta) =  \bigotimes_{i\in Q_0} \det(T_i)^{\otimes \eta_i}\quad\text{for }\eta=(\eta_i)\in \Theta.
 \end{equation}
  
 \begin{corollary}
 \label{cor:tilting}
 Under Assumption~\ref{ass:KrullSchmidt}, 
 the linearisation map $L_C$ is surjective.  
 \end{corollary}
 \begin{proof}
 The homomorphism $\det\colon K(X)\to \NS(X)$ induced by sending each locally-free sheaf to its determinant line bundle is surjective. Remark~\ref{rems:phipsi}(3) gives $K(X)= \bigoplus_{i\in Q_0} \ZZ[T_i^\vee]$, so the classes of the line bundles $\det(T_i^\vee)\cong \det(T_i)^{-1}$ for $i\in Q_0$ generate $\NS(X)$.
 \end{proof}

 \subsection{The numerical Grothendieck group}
 We now use the derived equivalence \eqref{eqn:derivedequivPhi} to transfer the results of Lemma~\ref{lem:NumKgroupNCCR} to the geometric side.
 
 \begin{lemma}
 \label{lem:NumKgroupX}
 \begin{enumerate}
     \item [\one] There is a $\ZZ$-linear isomorphism $\varphi^*\colon \Knum(A)\to \Knumc(X)$ with inverse $\psi^*:=(\varphi^*)^{-1}$ satisfying
   $\varphi^*\big([S_i]\big) = [\Psi(S_i)]$ for $i\in Q_0$. 
     \item[\two]  The isomorphisms from part \one\ identify $\ZZ v$ with $F_0$, inducing isomorphisms
   \begin{equation}
       \label{eqn:KnumgroupisomsmodF0}
 \begin{tikzcd}
\Knum(A)/\ZZ v \ar[r,shift left=0.5ex,"\varphi^*"] & \Knumc(X)/F_0.\ar[l,shift left=0.5ex,"\psi^*"]
\end{tikzcd}
  \end{equation}
     \item[\three] We have $\Knumc(X)\cong \bigoplus_{i\in Q_0} \ZZ[\Psi(S_i)]$, and the Euler form induces a perfect pairing 
     \[
     \chi\colon K(X)\times \Knumc(X)\to \ZZ.
     \]
 \end{enumerate}
 \end{lemma}
 \begin{proof}
  Following \cite[Theorem~7.2.1]{BCZ17}, the restriction of $\Phi$ to the full subcategory of objects with proper support $D^b_c(X)$ gives an equivalence of categories $\Phi\colon D^b_c(X)\to\Dfin(A)$ with quasi-inverse given by the restriction of $\Psi$. Moreover, \emph{ibid.\ }shows that the equivalences $\Phi$ and $\Psi$ induce the isomorphisms $\varphi^*$ and $\psi^*$ respectively from part \one\ (see the notational warning in Remark~\ref{rem:adjoint}), so $\varphi^*\big([S_i]\big)= \big[\Psi(S_i)\big]$ for $i\in Q_0$. This proves \one. For part \two, let $x\in X$ be a closed point. Then
 \begin{equation}
     \label{eqn:derivedOy}
  \Phi(\mathcal{O}_x) = \mathbf{R}\Gamma\circ  \mathbf{R}\mathcal{H}om(\mathcal{O}_{X},T\otimes \mathcal{O}_x) = \mathbf{R}\Gamma(T\otimes \mathcal{O}_x)= \Gamma(T\otimes \mathcal{O}_x)
 \end{equation}
 is the fibre of the tautological bundle $T$ over $x$, so the induced map on Grothendieck groups satisfies $\psi^*([\mathcal{O}_x]) = [\Gamma(T\otimes \mathcal{O}_x)] = v$. Since $X$ is connected, $[\mathcal{O}_x]$ spans $F_0$ over $\ZZ$, so $\psi^*$ identifies $F_0$ with $\ZZ v$, so the isomorphisms $\varphi^*$ and $\psi^*$ from part \one\ induce the isomorphisms \eqref{eqn:KnumgroupisomsmodF0} as required. Part \three\ follows by combining the isomorphisms in part \one\ and Remark~\ref{rems:phipsi}(3) with Lemma~\ref{lem:NumKgroupNCCR}.
  \end{proof}
  
  \begin{remark}
  \label{rem:adjoint}
  The isomorphisms $\varphi^*$ and $\psi^*$ from Lemma~\ref{lem:NumKgroupX}\one\ are adjoint (with respect to the Euler forms $\chi_A$ and $\chi$) to the isomorphisms $\varphi$ and $\psi$ from Remark~\ref{rems:phipsi}\three\ respectively, i.e.\ for $E\in K(A)$ and $F\in \Knumc(X)$, we have $\chi\big(\psi(E),F\big)  = \chi_A\big(E,\psi^*(F)\big)$, and similarly for $\varphi$ and $\varphi^*$. 
   \end{remark}

 \section{Gale duality for the linearisation map}
 \subsection{Gale duality over $\ZZ$}
 \label{sec:Galedual}
 We first recall the construction of the Gale dual for a homomorphism of finitely generated and free abelian groups. Let $\Lambda$ be one such group and let $\{\lambda_1, \dots, \lambda_n\}$ be a collection of elements that span $\Lambda$ over $\ZZ$; here, repetition is allowed. Define a group homomorphism $L\colon \ZZ^n\to \Lambda$ on the standard basis $\{\textbf{e}_1,\dots,\textbf{e}_n\}$ of $\ZZ^n$ by setting $L(\textbf{e}_i)=\lambda_i$ for $1\leq i\leq n$. We obtain a short exact sequence
 \[
\begin{tikzcd}
0 \ar[r] & \ker(L)\ar[r] & \ZZ^n\ar[r,"L"]&  \Lambda \ar[r]&  0. 
\end{tikzcd}
\]
Since $\Lambda$ is free, applying $\Hom_{\ZZ}(-,\ZZ)$ determines a dual short exact sequence
\[
\begin{tikzcd}
0  & \ker(L)^\vee\ar[l] & (\ZZ^n)^\vee \ar[l,swap,"G"] &  \Lambda^\vee\ar[l] &  0. \ar[l]
\end{tikzcd}
\]
If $\textbf{e}_1^\vee,\dots, \textbf{e}_n^\vee\in (\ZZ^n)^\vee$ is the dual $\ZZ$-basis, then the collection $\{G(\textbf{e}_1^\vee),\dots, G(\textbf{e}_n^\vee)\}$ in $\ker(L)^\vee$ is the \emph{Gale dual configuration} to $\{\lambda_1, \dots, \lambda_n\}$ in $\Lambda$ and we call $G$ the \emph{Gale dual map} to $L$. 

 Since $\ker(L)^\vee$ is also free, it follows similarly that the collection $\{\lambda_1,\dots,\lambda_n\}$ is the Gale dual configuration to $\{G(\textbf{e}^\vee_1),\dots, G(\textbf{e}^\vee_n)\}$ in $\ker(L)^\vee$ and that $L$ is the Gale dual map to $G$.
 
\subsection{The linearisation map via a filtration}
 Let $A$ be an NCCR of of threefold $\kk$-algebra $R$ as above. Suppose in addition that $A$ satisfies Assumption~\ref{ass:KrullSchmidt}, and that $v$ is indivisible. Then for any GIT chamber $C\subseteq \Theta_{\QQ}$ and for any $\theta\in C$, the fine moduli space $X=\mathcal{M}_\theta(A,v)$ and its tautological bundle $T$ satisfy Theorem~\ref{thm:BKR}. In particular, $X$ is a connected and separated scheme of finite type over $\kk$ that is proper over an affine base, so the results of Section~\ref{sec:numGrothGroups} apply.
    
  We now generalise Craw--Ishii~\cite[Section~5.1]{CI04}, but here we work over $\ZZ$ rather than $\QQ$. 
    
  \begin{lemma}
  \label{lem:CIrevisited}
  There is a commutative diagram of finitely generated and free abelian groups
   \begin{equation}
       \label{eqn:LCdiagram}
 \begin{tikzcd}
0 \ar[r] & \ker(L_C)\ar[r]\ar[d] & \Theta\ar[d,"\psi"]\ar[r,"L_C"]&  \NS(X) \ar[r]&  0 \\
0 \ar[r] & F^2\ar[r] & F^1\ar[r]&  F^1/F^2 \ar[u,swap,"\det(-)^{-1}"]\ar[r]&  0
\end{tikzcd}
  \end{equation}
 with exact rows, where the vertical maps are isomorphisms. 
\end{lemma}
 \begin{proof}
 The theorem of the base \cite[Theorem~1.1.16]{Lazarsfeld} shows $\NS(X)$ is free, and it is finitely generated by Corollary~\ref{cor:tilting} (this also follows from the proof of Corollary~\ref{cor:NS}). Note that $\Theta$ is also finitely generated and free, because $v$ is primitive, and hence so too is $\ker(L_C)$. The rows are exact by Corollary~\ref{cor:tilting} and by construction. For commutativity, write $\eta=\sum_{i\in Q_0} \eta_i[P_i]\in \Theta$ and compute 
 \begin{equation}
     \label{eqn:detLCinverse} 
 \det\big(\psi(\eta)\big) = \det\bigg(\sum_{i\in Q_0} \eta_i[T_i^\vee]\bigg) = \bigotimes_{i\in Q_0} \det(T_i)^{-\otimes\eta_i} = L_C(\eta)^{-1}.
  \end{equation}
 It remains to show that the vertical maps are isomorphisms. The proof of Lemma~\ref{lem:NumKgroupX}\two\ establishes that $\psi^*([\mathcal{O}_x])= v$, so an element $\theta\in K(A)$ lies in $\Theta$ if and only if 
 \[
 0 = \theta(v)=\chi_A(\theta,v) = \chi_A\big(\theta,\psi^*([\mathcal{O}_x]\big) = \chi\big(\psi(\theta),[\mathcal{O}_x]\big).
 \]
 Since $[\mathcal{O}_x]$ spans $F_0$, we have that $\theta\in \Theta$ if and only if $\psi(\theta)\in F_0^\perp=F^1$. Therefore, the restriction of $\psi$ identifies $\Theta$ with $F^1$, and Proposition~\ref{prop:Fmperp}\three\ gives that $\det(-)^{-1}$ is an isomorphism. \end{proof}

  \begin{lemma}
  \label{lem:CIdual}
     Let $G_C$ denote the map that is Gale dual to $L_C$. Then $G_C$ and the map $\psi^*$ from Lemma~\ref{lem:NumKgroupX}~\two\ fit into a commutative diagram of abelian groups 
  \begin{equation}
      \label{eqn:GCdiagram}
 \begin{tikzcd}
0 & \ker(L_C)^\vee\ar[l] & \Knum(A)/\ZZ v\ar[l,swap,"G_C"]&  \NS(X)^\vee\ar[d] \ar[l]&  0\ar[l]\\
0 & F_2/F_1\ar[l]\ar[u] & F_2/F_0\ar[l] \ar[u,"\psi^*"]&  F_1/F_0 \ar[l]&  0\ar[l],
\end{tikzcd}
  \end{equation}
 where the rows are exact and the vertical maps are isomorphisms.
  \end{lemma}
  \begin{proof}
  We simply apply $\Hom(-,\ZZ)$ to the diagram \eqref{eqn:LCdiagram}, but there are several things to check. First, the equality $\Knumc(X)=F_2$ follows from Corollary~\ref{cor:GensFm} because the maximal dimension of a proper subvariety in $X$ is 2. Proposition~\ref{prop:Fmperp}\two\ implies that $(F^1)^\vee\cong F_2/F_0$ and $(F^2)^\vee\cong F_2/F_1$, so we obtain the short exact sequence along the bottom of \eqref{eqn:GCdiagram}.  We have $\Theta^\vee \cong \Knum(A)/\ZZ v$, the sequence along the top is exact because $\NS(X)$ is free, and $G_C$ is Gale dual to $L_C$. Finally, by applying $\Hom(-,\ZZ)=\chi_A(-,\ZZ)$ to $\Theta$ and $\Hom(-,\ZZ)=\chi(-,\ZZ)$ to $F^1$, the adjoint property of $\psi$ and $\psi^*$ discussed in Remark~\ref{rem:adjoint} implies that $\psi^*=\Hom(\psi,\ZZ)$.
    \end{proof}

 \subsection{The Gale dual map via a filtration}
 \label{sec:Galedualviafiltration}
 The isomorphism $\psi^*$ in the diagram \eqref{eqn:GCdiagram}, as well as the isomorphism $F_2/F_1\to \ker(L_C)^\vee$ obtained from $\psi^*$, allows us to identify the map $G_C$ with the map that appears along the bottom-left of diagram \eqref{eqn:GCdiagram}. From now on, we make this identification, i.e.\ we identify the Gale dual of $L_C$ with the map 
 \begin{equation}
     \label{eqn:GCmap}
 G_C\colon F_2/F_0\longrightarrow F_2/F_1
  \end{equation}
 from diagram \eqref{eqn:GCdiagram}.  
 
 The fact that $X$ has dimension three allows us to describe $G_C$ explicitly. 
 
 \begin{proof}[Proof of Theorem~\ref{thm:GCintro}]
 For part \one, Lemma~\ref{lem:NumKgroupX}\three\ shows that the classes $[\Psi(S_i)]$ for $i\in Q_0$ provide a $\ZZ$-basis of the group $\Knumc(X)=F_2$, so they generate the quotient $F_2/F_0$. Now Lemma~\ref{lem:NumKgroupX}\two\ tells us that they satisfy the single relation $0=\sum_{i\in Q_0} v_i[\Psi(S_i)]$ in $F_2/F_0$. For part \two, we have $\Psi(S_i)\in D^b_c(X)$ for all $i\in Q_0$ because  $S_i\in \Dfin(A)$. Since each irreducible component of $\supp(\Psi(S_i))$ is proper, every such component is contracted to a point by the birational morphism $f\colon X\to \Spec R$. Therefore, the maximal dimension of any irreducible component of $\supp(\Psi(S_i))$ is $m=2$. Proposition~\ref{prop:Knumclass}
  then gives the class in $F_2$ as 
  \[
 \big[\Psi(S_i)\big] =\sum_{k\in \ZZ} (-1)^k \sum_{V} \ell_V\Big(\mathcal{H}^k\big(\Psi(S_i)\big)\Big) [\mathcal{O}_V] \mod F_{1},
 \]
 where the second sum is taken over all irreducible surfaces in the support of $\mathcal{H}^k\big(\Psi(S_i))$. The result follows because $G_C$ sends an element of $F_2/F_0$ to its class in $F_2/F_1$.
 \end{proof}

\subsection{Generalising the Cautis--Logvinenko conjecture}
\label{sec:property}
Before exploring the consequences of Theorem~\ref{thm:GCintro}, we consider the assumption on $v$ that we anticipate for Property~\ref{property} to hold. 

The original statement of Property~\ref{property} was proposed by Cautis--Logvinenko~\cite[Conjecture~1.2]{CL09} in the special case where $X$ is the $\Gamma$-Hilbert scheme of a finite subgroup $\Gamma\subset \SL(3,\kk)$ and $\Psi$ is the equivalence of derived categories considered by \cite{BKR01}. Example~\ref{exa:ghilb} shows that the skew group algebra of $\Gamma$ is Morita equivalent to the algebra $A:=\End_R(\bigoplus_{\rho} M_\rho)$ that satisfies Assumption~\ref{ass:KrullSchmidt}, and the desired dimension vector $v$ satisfies $v_i=\dim(\rho_i)$, where $\{\rho_0, \rho_1, \dots, \rho_r\}$ is the set of isomorphism classes of irreducible representations of $\Gamma$ with $\rho_0$ trivial. Notice that $v_0=1$ and, moreover, the tautological bundle on the $\Gamma$-Hilbert scheme satisfies $T_0\cong \mathcal{O}_X$. Cautis--Logvinenko~\cite{CL09} proved this conjecture when $\Gamma$ is abelian and $\mathbb{A}^3/\Gamma$ has an isolated singularity.

 More generally, for any CY3 toric algebra $A$ (see  Example~\ref{exa:dimers}) and $v=(1,\dots, 1)$, Property~\ref{property} was established by the author with Bocklandt and Quintero-V\'{e}lez~\cite{BCQ15}. For this, we distinguished one vertex $0\in Q_0$ and defined $C_+\subseteq \Theta_{\QQ}$ to be the GIT chamber containing the open cone
\begin{equation}
    \label{eqn:C+}
\big\{\theta\in \Theta_{\QQ} \mid \theta_i>0\text{ for }i\neq 0\big\}.
\end{equation}
 For any $\theta\in C_+$, the tautological bundle $T=\bigoplus_{i\in Q_0} T_i$ on the fine moduli space $X:=\mathcal{M}_\theta(A,v)$ of $\theta$-stable $A$-modules of dimension vector $v$ is well-defined only up to tensor product by a line bundle, and the choice of vertex $0$ enabled us to define $T$ uniquely by setting $T_0:=\mathcal{O}_{X}$. In addition, since $v_0=1$, Theorem~\ref{thm:GCintro}\one\ implies that the classes $[\Psi(S_i)]$ for $i\neq 0$ provide a $\ZZ$-basis of $F_2/F_0$.
 
 \begin{remark}
 \label{rem:6123}
 Property~\ref{property} makes no assertion about the cohomology sheaves of the object $\Psi(S_0)$. Indeed, for the group action of type $\frac{1}{6}(1,2,3)$ with $X= \Gamma\text{-Hilb}(\mathbb{A}^3)$ just as in Example~\ref{exa:1315},  we have $\mathcal{H}^{k}(\Psi(S_0))\neq 0$ for both $k=-1, -2$, see \cite{CL14}. Thus, the choice of the distinguished vertex $0\in Q_0$ plays an essential role in the study of Property~\ref{property}, even in the context of the original conjecture of Cautis--Logvinenko~\cite[Conjecture~1.2]{CL09}, because it identifies the unique object of the form $\Psi(S_i)$ that is not (the shift of) a sheaf.
 \end{remark}

 In Conjecture~\ref{conj:CLintro}, we propose a generalisation of Cautis--Logvinenko~\cite[Conjecture~1.2]{CL09} for an NCCR of $R$ of the form $A\cong \kk Q/I$ satisfying Assumption~\ref{ass:KrullSchmidt}. Rather than assume $v$ is indivisible, we impose the stronger condition that there is a distinguished vertex $0\in Q_0$ such that $v_0=1$. This choice determines the chamber $C_+$, and moreover, for every $\theta\in C_+$, it allows us to define the tautological bundle $T$ on the fine moduli space $X:=\mathcal{M}_\theta(A,v)$ uniquely by setting $T_0:= \mathcal{O}_{X}$. In fact, without this stronger assumption on $v$, the conjecture would have to encode in some other way how to define $T$ uniquely, and also, how the objects of the form $\Psi(S_i)$ that do not satisfy the conjecture - such as $\Psi(S_0)$ in the example from Remark~\ref{rem:6123} -  should be excluded.
 
 The evidence that Property~\ref{property} holds for non-toric algebras is sparse, but we can say the following. 
 
 \begin{example}
 \label{exa:HomMMP}
 Let $R$ denote a complete local normal Gorenstein ring for which a projective crepant birational morphism $f\colon X\to \Spec R$ exists with fibres of dimension at most one. Let $C_1, \dots, C_n$ denote the irreducible components (with reduced scheme structure) of the preimage of the unique closed point of $\Spec R$.  Van den Bergh~\cite{VdB04} constructs a tilting bundle $\mathcal{V}_X:=\mathcal{O}_X\oplus \bigoplus_{1\leq i\leq n} \mathcal{N}_i$ on $X$ such that for $A:=\End(\mathcal{V}_X^\vee) = \End (\mathcal{V}_X)^{\text{op}}$, the functor
  \[
     \Psi(-):= \mathcal{V}_X\ltensor_A(-) \colon D^b(A)\longrightarrow D^b(X)
     \]
     is a derived equivalence that sends the simple $A$-module $S_i$ to the sheaf $\mathcal{O}_{C_i}(-1)$. Note in particular that $\mathcal{H}^k(\Psi(S_i)) \neq 0$ only for $k=0$ as $\mathcal{O}_{\mathbb{P}^1}(-1)$ is a sheaf. For $v$ satisfying $v_0=1$ and $v_i=\rank(\mathcal{N}_i)$ for $i\neq 0$, Karmazyn~\cite[Corollary~5.2.4]{Karmazyn17} constructed an isomorphism $X\cong\mathcal{M}_\theta(A,v)$ for $\theta\in C_+$ that identifies the tautological bundle with $\mathcal{V}_X^{\vee}$. Thus, $\Psi$ agrees with \eqref{eqn:derivedequivalenceadjoint}, so Property~\ref{property} holds. Note that $A$ is an NCCR when $X$ is non-singular, so Conjecture~\ref{conj:CLintro} also holds in this case by the description of $C_+$ by Wemyss~\cite[Lemma~5.1]{Wemyss18}. See Section~\ref{sec:MMA} for the case when $X$ is singular.
 \end{example}
 
  \begin{remark}
  The heart $\mathcal{A}$ in $D^b(\coh(X))$ obtained as the essential image of $\modA$ under $\Psi$ is not simply a one-step tilt of $\coh(X)$, because $\mathcal{H}^{-2}(\Psi(S_0))\neq 0$ when $A$ is a CY-3 toric order and $X$ contains a proper surface; see \cite[Corollary~3.8]{BCQ15}. It is worth noting, however, that $\mathcal{A}$ is obtained via a two-step tilt from $\coh(X)$ when $A$ is the skew group algebra of a finite subgroup of $\SL(3,\kk)$ as in Example~\ref{exa:ghilb}, see Brown--Shipman~\cite{BrownShipman17}. Exploring this two-step tilt further might provide an approach to the original conjecture of Cautis--Logvinenko, and perhaps also Conjecture~\ref{conj:CLintro}\one.
    \end{remark}

\subsection{The sign-coherence property}
\label{sec:signcoherence}
 The power of Property~\ref{property} lies in the implicit vanishing statement which forces the first sum from \eqref{eqn:GCintro} to collapse. This leads immediately to the following.
 
 \begin{proposition}
 \label{prop:GC}
 Let $A\cong \kk Q/I$ be an NCCR of the ring $R$ satisfying Assumption~\ref{ass:KrullSchmidt}, and assume $v_0=1$ for some $0\in Q_0$. Set $X:= \mathcal{M}_\theta(A,v)$ for $\theta\in C_+$.  If Property~\ref{property} holds, then for each $i\neq 0$, we have 
  \begin{equation}
     \label{eqn:GCmain}
 G_{\chamber}\big([\Psi(S_i)]\big) = (-1)^{k(i)} \sum_{V} \ell_V\Big(\mathcal{H}^{k(i)}\big(\Psi(S_i)\big)\Big)[\mathcal{O}_{V}],
 \end{equation}
 where the sum runs over all irreducible surfaces in the support of $\mathcal{H}^{k(i)}\big(\Psi(S_i)\big)$.
 \end{proposition}

Recall from Definition~\ref{def:SC} that a matrix is \emph{sign-coherent} if for each column, the nonzero entries are either all positive or all negative (compare Fomin--Zelevinsky~\cite[Definition~6.12]{FZ07}).  To define the matrix of interest,  Corollary~\ref{cor:GensFm} implies that the irreducible components $E_1, \dots, E_\ell$ of the exceptional divisor of $f_\theta\colon X\to \Spec R$ define classes $[\mathcal{O}_{E_1}], \dots , [\mathcal{O}_{E_\ell}]$ generating $F_2/F_1$ over $\ZZ$.

 \begin{corollary}
 \label{cor:GC}
 Under the assumptions of Proposition~\ref{prop:GC}, and suppose too that $[\mathcal{O}_{E_1}], \dots , [\mathcal{O}_{E_\ell}]$ provide a $\ZZ$-basis of $F_2/F_1$. Then  $G_C\colon F_2/F_0\to F_2/F_1$ is represented by a sign-coherent matrix.
 \end{corollary}
 \begin{proof}
 Consider the $\ZZ$-basis $\{[\Psi(S_i)] \mid i\neq 0\}$ of $F_2/F_0$ from Theorem~\ref{thm:GCintro}\one\ and the given basis of $F_2/F_1$. Then \eqref{eqn:GCmain} shows that the components of the $i^{th}$ column of the matrix representing $G_C$ in these bases are all non-negative if $k(i)$ is even, and non-positive if $k(i)$ is odd. 
 \end{proof}
 
 We already saw this attractive phenomenon in Example~\ref{exa:1315},  and we illustrate it further in Examples~\ref{exa:longhex} and \ref{exa:BentoBox}.

 \subsection{Maximal Modification Algebras}
 \label{sec:MMA}
 We established that Property~\ref{property} holds for the algebra $A$ from Example~\ref{exa:HomMMPcont} in the special case when the quiver moduli space $X=\mathcal{M}_\theta(A,v)$ for $\theta\in C_+$ is non-singular. When $X$ is singular, $A$ is not an NCCR \cite[Theorem~1.5(2)]{IyamaWemyss14b} and therefore we cannot claim to have proven Conjecture~\ref{conj:CLintro}\one\, even though Property~\ref{property} holds.
 
 To address this problem, recall from Iyama--Wemyss~\cite{IyamaWemyss14} that a reflexive $R$-module $M$ is called a  \emph{modifying module} if $\End_R(M)$ is Cohen--Macaulay as an $R$-module, and $M$ is a \emph{maximal modifying module} if it is maximal with this property. A \emph{maximal modification algebra} (MMA) is any algebra of the form $\End_R(M)$, where $M$ is a maximal modifying module. If $R$ admits an NCCR, then every NCCR is an MMA \cite[Proposition~4.5]{IyamaWemyss14}. More generally,  \cite[Theorem~4.16]{IyamaWemyss14b} makes a compelling case that in order to probe the geometry of threefold minimal models that are derived equivalent to an algebra $A$, one should work with MMAs rather than NCCRs whenever possible.

 We expect that Theorem~\ref{thm:GCintro} and Conjecture~\ref{conj:CLintro} extend to any MMA  satisfying Assumption~\ref{ass:KrullSchmidt}. However, we chose in the end to formulate Theorem~\ref{thm:GCintro} and Conjecture~\ref{conj:CLintro} only for an NCCR simply because our understanding of the quiver moduli spaces $\mathcal{M}_\theta(A,v)$ is currently too limited when $A$ does not have finite global dimension. The only exception known to the author is Example~\ref{exa:HomMMP}, where $G_C$ is the zero map (simply because $F_2/F_1=0$) and Property~\ref{property} is already known to hold. Nevertheless, for completeness we remark that, if one is prepared to live with a number of strong assumptions about $\mathcal{M}_\theta(A,v)$ and its tautological bundle, then the analogue of Theorem~\ref{thm:GCintro} for an MMA can be stated as follows.

 \begin{proposition}
 \label{prop:GCMMA}
 Let $R$ be a normal, Gorenstein finitely generated $\kk$-algebra of dimension three and let $A$ be an MMA of $R$ satisfying Assumption~\ref{ass:KrullSchmidt} such that $v$ is indivisible. Let $C\subset \Theta\otimes_{\ZZ} \QQ$ be a GIT chamber such that for $\theta\in C$, the following two conditions are satisfied:
\begin{enumerate}
    \item[\one] there is a projective birational morphism $f_\theta\colon \mathcal{M}_\theta(A,v)\to \Spec R$; and
    \item[\two] the tautological bundle $T$ on $\mathcal{M}_\theta(A,v)$ is a tilting bundle, giving a derived equivalence
\[
\Psi:=T^\vee\otimes_A (-) \colon D^b(A)\longrightarrow D^b\big(\mathcal{M}_\theta(A,v)\big). 
\]
\end{enumerate}
 Then the map $G_C\colon F_2/F_0\to F_2/F_1$ that is Gale dual to the linearisation map $L_C$ is determined by
 \begin{equation}
     \label{eqn:GCMMA}
 G_{\chamber}\big([\Psi(S_i)]\big) = \sum_{k\in \ZZ} (-1)^k \sum_{V} \ell_V\Big(\mathcal{H}^k\big(\Psi(S_i)\big)\Big)[\mathcal{O}_{V}]
 \end{equation}
 for $i\in Q_0$, where the second sum is over all irreducible surfaces in the support of $\mathcal{H}^k\big(\Psi(S_i))$.
\end{proposition}
\begin{proof}[Sketch proof]
 The morphism $f_\theta$ is crepant by \cite[Theorem~1.5(1)]{IyamaWemyss14b} because $\mathcal{M}_\theta(A,v)$ is derived equivalent to $A$. The assumptions allow us to follow the same proof as for Theorem~\ref{thm:GCintro}, except in that one has to circumvent our use of the finite global dimension of $A$ (see Lemma~\ref{lem:NumKgroupNCCR}) by replacing $K(A)$ throughout by the subgroup of $K(A)$ generated by projective $A$-modules (one also loses non-singularity in Theorem~\ref{thm:BKR}\two, but that doesn't play a role in the proof of Theorem~\ref{thm:GCintro}). 
\end{proof}

 \begin{example}
 \label{exa:HomMMPcontcont}
 Continuing Example~\ref{exa:HomMMP}, if $X$ is singular, then $A$ does not have finite global dimension, so it is not an NCCR. However, $f\colon X=\mathcal{M}_\theta(A,v)\to \Spec R$ is a projective, crepant birational morphism, so $A$ is an MMA by \cite[Theorem~1.5(1)]{IyamaWemyss14b}. The assumptions of Proposition~\ref{prop:GCMMA} hold \cite{Karmazyn17}, though we already know $G_C$ is the zero map, and we saw above that Property~\ref{property} holds. Put simply, the assumption that $X$ is non-singular appearing towards the end of Example~\ref{exa:HomMMP} is a red herring.
  \end{example}

\section{Comparing rival versions of Reid's recipe}
  We now restrict attention to the case where $A$ is a toric algebra and $v=(1,\dots,1)$. We recall Reid's recipe and the geometric interpretation of Reid~\cite{Reid97}, as well as the alternative geometric interpretation due to Cautis--Logvinenko~\cite{CL09}. As an application of Corollary~\ref{cor:GC}, we establish that these rival geometric interpretations are equivalent. To conclude, we establish that two conjectures in the general toric case are equivalent.
 
 \subsection{Decorating the toric fan}
  Suppose that $R=\kk[x,y,z]^\Gamma$ for a finite abelian subgroup $\Gamma\subset \SL(3,\kk)$, in which case the skew group algebra $A:= \kk[x,y,z]\star \Gamma$ is a $CY$-$3$ toric order, i.e.\ we are working in the situation of Examples~\ref{exa:ghilb} and~\ref{exa:dimers} simultaneously. The McKay quiver with relations gives a presentation $A\cong \kk Q/I$ satisfying Assumption~\ref{ass:KrullSchmidt}, and we let $0\in Q_0$ denote the vertex corresponding to the trivial representation of $\Gamma$. 
 For $v=(1,\dots,1)$ and for the chamber $C_+\subset \Theta_{\QQ}$ containing the open cone \eqref{eqn:C+}, the projective crepant toric resolution 
 \[
 f\colon X=\mathcal{M}_\theta(A,v)\longrightarrow \mathbb{A}^3/\Gamma=\Spec R
 \]
 from Theorem~\ref{thm:BKR} is the $\Gamma$-Hilbert scheme of Nakamura~\cite{Nakamura01}. The summands of the tautological bundle $T=\bigoplus_{0\leq i\leq 18} T_i$ are globally generated line bundles with $T_0\cong \mathcal{O}_X$.  
 
 Let $\Sigma$ denote the fan of the toric variety $X=\mathcal{M}_\theta(A,v)$, and write $\Delta$ for the height-one slice of $\Sigma$. Reid's combinatorial recipe for marking 1- and 2-dimensional cones of the fan $\Sigma$ (equivalently, lattice points and lines segments of the height-one slice $\Delta$) with vertices of the quiver $Q$ is summarised in \cite[Section~2.3]{CCL17}. Rather than describe the recipe in detail here, we refer the reader to the enlightening example of $\frac{1}{19}(1,3,15)$ in the Introduction.
 
\subsection{Reconciling rival geometric interpretations}
The original, geometric construction that motivated Reid~\cite{Reid97} to introduce his recipe can be summarised as follows:

 \begin{definition}[\textbf{Interpreting Reid's recipe - version 1}]
 \label{def:RR1}
 The combinatorics of Reid's recipe encodes the following geometric information: 
 \begin{enumerate}
    \item[\one] a $\ZZ$-basis of $\NS(X)$ given by a subset of the tautological bundles; and
    \item[\two] a set of minimal relations in $\NS(X)$ between all of the tautological bundles.
\end{enumerate}
Full details of the choice of basis and the construction of the relations appear in \cite{Craw05}. 
\end{definition}

\begin{proposition}
\label{prop:RR1}
The geometric interpretation of Reid's recipe given in Definition~\ref{def:RR1} is encoded by matrices that represent the maps in the short exact sequence
 \begin{equation}
 \begin{tikzcd}
0 \ar[r] & \ker(L_{C_+})\ar[r,"K"] & \Theta\ar[r,"L"]&  \NS(X) \ar[r]&  0
\end{tikzcd}
  \end{equation}
  appearing in diagram \eqref{eqn:LCdiagram}.
\end{proposition}
\begin{proof}
 Since $v_0=1$, choose $\{[P_i]\mid i\neq 0\}$ as a $\ZZ$-basis of  $\Theta$. Corollary~\ref{cor:tilting} shows that the bundles $T_i=L_{C_+}([P_i])$ for $i\neq 0$ generate $\NS(X)$ over $\ZZ$. Reid's recipe \cite[Proposition~7.1]{Craw05} chooses the $\ZZ$-basis for $\NS(X)$ comprising the bundles $T_i$ indexed by vertices marking line segments in the height-one slice $\Delta$ of the toric fan, as well as one additional $T_j$ for each lattice point in $\Sigma$ lying at the intersection point of three straight lines (every such lattice point is marked with two vertices; pick $j$ to be either one). For $i\neq 0$, the $i^{th}$ column of $L$ records the coefficients of $T_i$ in this basis. 
 
 To define the matrix $K$, note that the recipe associates a minimal relation between the tautological bundles for each interior lattice point in $\Delta$, where the indices in each relation are those vertices marking the lattice point itself and line segments that lie adjacent to the lattice point. List the interior lattice points, and for the $j^{th}$ such point, the $j^{th}$ column of $K$ records the exponents of the bundles $T_i$ appearing in the relation (after clearing all line bundles from the right-hand side of the equation, thereby introducing some negative exponents on the left). The relations are minimal in $\NS(X)$ by \cite[Proposition~6.3]{Craw05}, so $K$ represents the inclusion of $\ker(L_{C_+})$ into $\Theta$ as required.
\end{proof}

 A rival geometric interpretation of Reid's combinatorial recipe was established for an isolated singularity in the work of  Cautis--Logvinenko~\cite{CL09} and Logvinenko~\cite{Logvinenko10}, and subsequently by Cautis--Craw--Logvinenko~\cite{CCL17} for any finite abelian subgroup of $\SL(3,\kk)$.
 
 \begin{definition}[\textbf{Interpreting Reid's recipe - version 2}]
 \label{def:RR2}
 The combinatorics of Reid's recipe encodes the following geometric information: 
 \begin{enumerate}
     \item[\one] for each $i\neq 0$ such that each component of $\supp \Psi(S_i)$ has dimension two, either:
     \begin{itemize}
         \item $k(i)=-1$, and $\supp \Psi(S_i)$ is the union of surfaces defined by lattice points in $\Delta$ such that two or more line segments labelled $i$ that lie adjacent to the lattice point; or
         \item $k(i)=0$, and $\supp \Psi(S_i)$ is the surface defined by the lattice point in $\Delta$ labelled $i$.
     \end{itemize}
       \item[\two]  otherwise, for $i\neq 0$ we have $k(i)=0$, and $\supp \Psi(S_i)$ is the curve defined by the unique line segment in $\Delta$ marked $i$.
 \end{enumerate}
 
 \end{definition}

\begin{remark}
\label{rem:DRR}
 Logvinenko~\cite{Logvinenko10} provides a slightly more precise interpretation than Definition~\ref{def:RR2}, in that he computes explicitly the objects $\Psi(S_i)$ when $k(i)=0$. This was strengthened further in~\cite{CCL17} to `\emph{Derived Reid's recipe}' for a finite abelian subgroup $\Gamma\subset \SL(3,\kk)$, where the objects $\Psi(S_i)$ were computed explicitly for all $i\in Q_0$. 
\end{remark}

 It turns out that the data from Definition~\ref{def:RR2} can also be encoded efficiently by two matrices, each obtained as the transpose of one of the matrices from Proposition~\ref{prop:RR1}.
 
\begin{proposition}
\label{prop:RR2}
The geometric interpretation of Reid's recipe given in Definition~\ref{def:RR2} is encoded by the matrices $K^t$ and $L^t$ that represent the maps in the short exact sequence
 \begin{equation}
 \label{eqn:KtLTseq}
 \begin{tikzcd}
0 & F_2/F_1\ar[l] & F_2/F_0\ar[l,swap,"K^t"] &  F_1/F_0 \ar[l,swap,"L^t"]&  0\ar[l]
\end{tikzcd}
  \end{equation}
  appearing in the diagram \eqref{eqn:GCdiagram}.
\end{proposition}
\begin{proof}
 Property~\ref{property} holds in this case, so Proposition~\ref{prop:GC} shows that the $i^{\text{th}}$ column of $K^t$ records the length of $\mathcal{H}^{k(i)}(\Psi(S_i))$ on each exceptional divisor $E_j$ in $X$. There are three cases, according to the sign of the entries in the $i^{\text{th}}$ column of $K^t$ as follows.
 \begin{enumerate}
     \item[($+$)] If each nonzero entry is positive, then $\supp(\mathcal{H}^0(\Psi(S_i)))$ contains a surface $E_j$ for some $j$. Then \cite[Theorem~1.2]{CCL17} implies that $\supp(\Psi(S_i))=E_j$, where vertex $i$ labels the lattice point in $\Delta$ defining $E_j$ by Reid's recipe. 
     \item[($0$)] If each entry is zero, then $[\Psi(S_i)]\in \ker(G_{C_+})=F_1/F_0$, so $\Psi(S_i)$ is supported on curves. Now \cite[Theorem~1.2]{CCL17} shows that  $\supp(\Psi(S_i))=\ell_j$, where vertex $i$ marks the line segment in $\Delta$ defining $\ell_j\cong \mathbb{P}^1$ by Reid's recipe. 
        \item[($-$)] If each nonzero entry is negative, then $\supp(\mathcal{H}^{-1}(\Psi(S_i)))$ contains a surface $E_j$ for some $j$. Again,  \cite[Theorem~1.2]{CCL17} implies that $\supp \Psi(S_i)$ is the union of surfaces $E_j$ defined by lattice points in $\Delta$ such that two or more line segments  marked $i$ lie adjacent to the point. 
 \end{enumerate}
 This completes the proof.
\end{proof}

\begin{remarks}
\begin{enumerate}
\item Property~\ref{property} holds in this case \cite{CL09, CCL17}, so $K^t$ in \eqref{eqn:KtLTseq} is sign-coherent.
\item It is possible to strengthen Definition~\ref{def:RR1} and \ref{def:RR2} slightly by making reference to the entries of $K$ and $K^t$ respectively. We make this explicit in Theorem~\ref{thm:CHT} below.
\end{enumerate}
\end{remarks}

\begin{theorem}
\label{thm:reconciliation}
 The rival geometric interpretations of Reid's recipe stated in Definitions~\ref{def:RR1} and \ref{def:RR2} are equivalent.
\end{theorem}
\begin{proof}
 Given Propositions~\ref{prop:RR1} and \ref{prop:RR2}, this is nothing more than the statement that knowing the matrices $K$ and $L$ is equivalent to knowing the transpose matrices $K^t$ and $L^t$.
\end{proof}

As noted in Example~\ref{exa:1315}, the support of any object $\Psi(S_i)$ for $i\neq 0$ can be read off immediately from the $i^{\text{th}}$ column of the matrix $K^t$ simply by invoking Proposition~\ref{prop:GC}. Note that $K^t$ can be written down immediately using only the relations listed in \cite[Theorem~6.1]{Craw05}.

 \subsection{Equivalence of two conjectures in the general toric case}
 Consider the general toric case, where $\Spec R$ is any Gorenstein toric threefold and $A\cong \kk Q/I$ is a noncommutative crepant resolution of $R$ as in Example~\ref{exa:dimers}. 
 
 In this context, work of the author with Heuberger and Tapia Amador~\cite{CHT21} considers the toric crepant resolution $X=\mathcal{M}_\theta(A,v)\to \Spec R$ for $\theta\in C_+$ from \eqref{eqn:C+}. In that paper, we introduce a recipe marking the height-one slice $\Delta$ of the toric fan $\Sigma$ of $X$ with nonzero vertices of $Q$ in a manner that generalises Reid's recipe and is compatible with the calculation of those objects $\Psi(S_i)$ indexed by nonzero vertices $i\in Q_0$ for which $k(i)=0$, see \cite[Theorems~1.1, 1.4]{BCQ15}. However, this does not provide a complete analogue of Definition~\ref{def:RR2} in this case, because computing even the supports of the objects $\Psi(S_i)$ for which $k(i)=-1$ is challenging. 
 
 Nevertheless, we formulated a pair of conjectures \cite[Conjectures~5.5 and 5.7]{CHT21} which together state that this recipe provides the appropriate generalisation of Reid's recipe from Definitions~\ref{def:RR1} and \ref{def:RR2}. To state the conjectures,  for each nonzero vertex $i\in Q_0$ and interior lattice point $\rho\in \Delta$, we define $n(i,\rho)$ to be the number of interior line segments in $\Delta$ marked by vertex $i$ that touch lattice point $\rho$. It is convenient to define $N(i,\rho):=\max\{0,n(i,\rho)-1\}$.

  \begin{conjecture}\cite[Conjecture~5.7]{CHT21} 
  \label{conj:RR1}
  For each lattice point $\rho\in \Delta$, the tautological line bundles on $X=\mathcal{M}_\theta(A,v)$ satisfy $
  \bigotimes_{i\text{ marks }\rho} T_{i}\cong\bigotimes_{i\in Q_0} T_{i}^{N(i,\rho)}.$
  \end{conjecture}

 \begin{conjecture}\cite[Conjecture~5.5]{CHT21} 
 \label{conj:RR2}
  For each lattice point $\rho\in \Delta$, let $E_\rho$ be the divisor associated to the ray in $\Sigma$ through $\rho$. For each vertex $i\neq 0$ and each $\rho\in \Delta$, we have $\ell_{E_\rho}\big(\mathcal{H}^{-1}(\Psi(S_i))\big)
     = N(i,\rho)$.
 \end{conjecture}

In fact, we have edited slightly the statement of both conjectures. Explicitly:
 \begin{enumerate}
     \item[\one] \cite[Conjecture~5.7]{CHT21} writes the exponent on the right hand side as $n(i,\rho)-1$, but we had in mind that $i$ marked at least one line segment touching $\rho$ (making $n(i,\rho)\geq 1$). I correct this omission by introducing $N(i,\rho)$ above.
       \item[\two] \cite[Conjecture~5.5]{CHT21} describes only the support of $\mathcal{H}^{-1}(\Psi(S_i))$, or equivalently, the union of surfaces $E_\rho$ for which $\ell_{E_\rho}(\mathcal{H}^{-1}(\Psi(S_i)))\neq 0$, in terms of Reid's recipe. In the above, I simply provide the numerical value for this length.
 \end{enumerate}
 Both conjectures remain open, but we can say the following.
 
 \begin{theorem}
 \label{thm:CHT}
 Let $A$ be a CY-3 toric order and $v=(1,\dots, 1)$. The conjectural rival geometric interpretations of Reid's recipe stated in Conjectures~\ref{conj:RR1} and \ref{conj:RR2} are equivalent.
  \end{theorem}
 \begin{proof}
 Fix an interior lattice point $\rho\in \Delta$. It suffices to prove that
  \[
 \ell_{E_\rho}(\mathcal{H}^{-1}(\Psi(S_i)))= N(i,\rho) \text{ for all }i\in Q_0\setminus \{0\}\iff 
 \bigotimes_{i\text{ marks }\rho} T_{i}\cong
\bigotimes_{i\in Q_0}
T_{i}^{N(i,\rho)}.
\]
Suppose the left hand side holds. Since Property~\ref{property} holds by \cite[Theorem~1.1]{BCQ15}, we deduce that the matrix $K^t$ has $(\rho,i)$ entry equal to $-N(i,\rho)$ whenever $k(i)=-1$. But we already know  by \cite[Proposition~5.2]{CHT21} that when $k(i)=0$,  the matrix $K^t$ has $(\rho,i)$ entry equal to $1$ if $i$ marks vertex $\rho$, and 0 otherwise. Taking the transpose, we see that the column of $K$ indexed by $\rho$ has an entry of 1 whenever $i$ marks vertex $\rho$, and an entry of $-N(i,\rho)$ when $k(i)=-1$. Thus, there is a relation
\[
\bigotimes_{i\text{ marks }\rho} T_{i}\otimes \bigotimes_{\{i\in Q_0 \mid k(i)=-1\}}
T_{i}^{-N(i,\rho)}\cong \mathcal{O}_X.
\]
 We obtain the desired relation because $N(i,\rho)=0$ for each $i$ satisfying $k(i)\neq -1$. The converse holds by reversing the logic of this argument.
 \end{proof}

\begin{example}
 \label{exa:longhex}
 Consider the toric algebra $A\cong \kk Q/I$ arising as the path algebra modulo relations of the consistent dimer model shown in \cite[Figure~2(a)]{CHT21}. The labelling of the height-one slice $\Delta$ of the toric fan $\Sigma$ by nonzero vertices of the quiver $Q$ is reproduced below.
\begin{figure}[!ht]
\centering
\begin{tikzpicture}[xscale=1.3,yscale=1.3] 
\draw (0,0) -- (0,1) -- (-1,2) -- (-2,3) -- (-3,3) -- (-3,2) -- (-2,1) -- (-1,0) -- (0,0); 
\draw (0,0) -- node [rectangle,draw,fill=white,sloped,inner sep=1pt] {{\footnotesize 9}}(-1,1) -- node [rectangle,draw,fill=white,sloped,inner sep=1pt] {{\footnotesize 3,9}}(-2,2) -- node [rectangle,draw,fill=white,sloped,inner sep=1pt] {{\footnotesize 3}}(-3,3); 
\draw (0,1) -- node [rectangle,draw,fill=white,sloped,inner sep=1pt] {{\footnotesize 7}}(-1,1) -- node [rectangle,draw,fill=white,sloped,inner sep=1pt] {{\footnotesize 4}}(-2,1); 
\draw (-1,2) -- node [rectangle,draw,fill=white,sloped,inner sep=1pt] {{\footnotesize 6}}(-2,2) -- node [rectangle,draw,fill=white,sloped,inner sep=1pt] {{\footnotesize 6}}(-3,2); 
\draw (-2,3) -- node [rectangle,draw,fill=white,sloped,inner sep=1pt,rotate=90] {{\footnotesize 8}}(-2,2) -- node [rectangle,draw,fill=white,sloped,inner sep=1pt,rotate=90] {{\footnotesize 8}}(-2,1); 
\draw (-1,2) -- node [rectangle,draw,fill=white,sloped,inner sep=1pt,rotate=90] {{\footnotesize 4}}(-1,1) -- node [rectangle,draw,fill=white,sloped,inner sep=1pt,rotate=90] {{\footnotesize 1}}(-1,0);
\draw (0,0) node[circle,draw,fill=white,minimum size=5pt,inner sep=1pt] {{}};
\draw (0,1) node[circle,draw,fill=white,minimum size=5pt,inner sep=1pt] {{}};
\draw (-2,3) node[circle,draw,fill=white,minimum size=5pt,inner sep=1pt] {{}};
\draw (-3,3) node[circle,draw,fill=white,minimum size=5pt,inner sep=1pt] {{}};
\draw (-3,2) node[circle,draw,fill=white,minimum size=5pt,inner sep=1pt] {{}};
\draw (-1,0) node[circle,draw,fill=white,minimum size=5pt,inner sep=1pt] {{}};
\draw (-2,1) node[circle,draw,fill=white,minimum size=5pt,inner sep=1pt] {{}};
\draw (-2,2) node[circle,draw,fill=white,minimum size=5pt,inner sep=1pt] {{\footnotesize 2, 5}};
\draw (-1,1) node[circle,draw,fill=white,minimum size=5pt,inner sep=1pt] {{\footnotesize 2}};
\draw (-1,2) node[circle,draw,fill=white,minimum size=5pt,inner sep=1pt] {{}};
\end{tikzpicture}
\caption{Combinatorial Reid's recipe for a dimer model example}
\label{fig:LongHexRR}
\end{figure}
It was computed in \emph{ibid}.\ that there are two minimal relations between the tautological line bundles, one for each interior lattice point in Figure~\ref{fig:LongHexRR}, namely $T_2\otimes T_5\cong T_3\otimes T_6\otimes T_8$ and $T_2\cong T_4\otimes T_9$. We compute that
 \[
\setlength{\arraycolsep}{2pt}
   \renewcommand{\arraystretch}{0.8}
K^t=\mbox{\footnotesize$\left(\begin{array}{ccrrcrcrr}
0 & 1 &  -1 &  0 & 1  & -1 & 0 & -1 & 0 \\
0 & 1 &  0 & -1 & 0  & 0 & 0 & 0 & -1    
\end{array}\right)$}
\]
 The relations have the form suggested by Conjecture~\ref{conj:RR1}, so Theorem~\ref{thm:CHT} shows that Conjecture~\ref{conj:RR2} holds too, thereby providing some geometric information about the objects $\Psi(S_i)$ satisfying $k(i)=-1$ purely in terms of the combinatorial recipe from \cite{CHT21}.
 \end{example}

\begin{remark}
 Tapia Amador~\cite{TapiaAmador15} contains many examples for which \cite[Conjecture~5.7]{CHT21} has been verified, so \cite[Conjecture~5.5]{CHT21} holds too.
\end{remark}
 
 \section{Reid's recipe for noncommutative crepant resolutions}
 We conclude by outlining what it is required in order to calculate Reid's recipe for any noncommutative crepant resolution $A$ satisfying Assumption~\ref{ass:KrullSchmidt} and $v$ satisfies $v_0=1$ for some $0\in Q_0$.

 \subsection{The statement}
 Theorem~\ref{thm:reconciliation} suggests how to formulate Reid's recipe beyond the toric case. For this,
 let $R$ be any finitely generated $\kk$-algebra that is normal, Gorenstein and of Krull dimension three. Suppose there is an NCCR
 $A\cong \kk Q/I$ of $R$ satisfying Assumption~\ref{ass:KrullSchmidt}, such that $v_0=1$ for some vertex $0\in Q_0$. Let $C_+$ be the GIT chamber containing the cone \eqref{eqn:C+}, and normalise the tautological bundle $T$ on $X=\mathcal{M}_\theta(A,v)$ for $\theta\in C_+$ so that $T_0\cong \mathcal{O}_X$.
 
 \begin{definition}
 \label{def:RR}
  For $\theta\in C_+$, \emph{Reid's recipe} for $X=\mathcal{M}_\theta(A,v)$ is the calculation of the matrices $K$ and $L$ representing the maps in the short exact sequence 
  \begin{equation}
 \begin{tikzcd}
0 \ar[r] & \ker(L)\ar[r] & \Theta\ar[r,"L"]&  \NS(X) \ar[r]&  0,
\end{tikzcd}
  \end{equation}
 or equivalently, the calculation of the transpose matrices $K^t$ and $L^t$ defining the dual sequence
  \begin{equation}
 \begin{tikzcd}
0 & F_2/F_1\ar[l] & F_2/F_0\ar[l,swap,"G"] &  F_1/F_0 \ar[l]&  0.\ar[l]
\end{tikzcd}
  \end{equation}
 \end{definition}

 \begin{remark}
     As a warning, the support of the objects $\Psi(S_i)$ for $i\neq 0$ can be computed directly from $K^t$ only when the assumptions of Corollary~\ref{cor:GC} hold. When this is the case, the minimal relations between the determinants of the tautological bundles take an especially nice form: for each $1\leq j\leq n$, there is a minimal relation of the form 
     \[
     \bigotimes_{\{i\in Q_0\mid k(i)=0\}} \det(T_i)^{K_{i,j}} \cong \bigotimes_{\{i\in Q_0\mid k(i)=-1\}} \det(T_i)^{-K_{i,j}}, 
     \]
     where $K_{i,j}$ is the $(i,j)$-entry of the matrix $K$.
   \end{remark}

 \subsection{Non-toric examples}
 We conclude by calculating the matrices $K$ and $L$ for the non-toric examples of Reid's recipe that appear in the literature.
 
 \begin{example}[\textbf{Homological MMP}]
 \label{exa:HomMMPcont}
 Continuing Example~\ref{exa:HomMMP}, combine the description of $\NS(X)$  from \cite[Lemma~3.4.3]{VdB04} with the moduli construction of  \cite[Corollary~5.2.4]{Karmazyn17} to see that the determinant line bundles $\det(T_i) = \det(\mathcal{N}_i^\vee) = \det(\mathcal{N}_i)^{-1}$ for $i\neq 0$ provide a $\ZZ$-basis of $\NS(X)$ dual to the classes of the exceptional curves $C_1, \dots, C_n$. Thus, we may choose bases such that $L$ is the identity matrix, so $K$ is the zero matrix. The dual statement is even clearer: $X$ contains no proper surfaces, so $F_2/F_1=0$, forcing $G=G_{C_+}$ to be the zero map and the inclusion $F_1/F_0\to F_2/F_0$ to be the isomorphism that sends the class of the curve $C_i$ to  $\Psi(S_i)=\mathcal{O}_{C_i}(-1)$ for $i\neq 0$. 
 \end{example}

 \begin{remark}
 The same conclusions as Examples~\ref{exa:HomMMP} and \ref{exa:HomMMPcont} can be drawn for $X=\Gamma\text{-Hilb}(\mathbb{A}^3)$ where  $\Gamma\subset \SO(3)$ is any finite subgroup; compare Nolla de Celis--Sekiya~\cite{NollaSekiya17}.
 \end{remark}
 
 \begin{example}[\textbf{Bento box}]
 \label{exa:BentoBox}
  Let $\Gamma\subset \SL(3,\kk)$ be the dihedral group of order 24 generated by 
 \[
 \begin{pmatrix} e^{\pi i/6} & 0 & 0 \\ 
 0 & e^{7\pi i/6} & 0 \\ 0 & 0 & e^{2\pi i/3} \end{pmatrix}\quad\text{and}\quad
  \begin{pmatrix} 
 0 & 1 & 0 \\ 
 -1 & 0 & 0 \\ 
 0 & 0 & 1
 \end{pmatrix}
 \]
 Following  
 Nolla de Celis~\cite{NolladeCelis21}, we list the 15 irreducible representations as
 \begin{equation}
    \label{eqn:irreps}
 \{\rho_{0^+}, \rho_{0^-}, \rho_{1}, \rho_{2^+}, \rho_{2^-}, \rho_{4^+}, \rho_{4^-}, \rho_{5}, \rho_{6^+}, \rho_{6^-}, \rho_{8^+}, \rho_{8^-}, \rho_{9}, \rho_{10^+}, \rho_{10^-}\},
  \end{equation}
 where $\rho_1, \rho_5, \rho_9$ have dimension two, every other irreducible representation has dimension one, and where $\rho_{0^+}$ is the trivial representation. The preprojective algebra $A$ of $\Gamma$ is an NCCR of $R=\kk[x,y,z]^\Gamma$. For each irreducible representation $\rho_i$ from the list \eqref{eqn:irreps}, write $T_i$ for the corresponding tautological bundle of rank $v_i:=\dim\rho_i$ on $X=\mathcal{M}_\theta(A,v)$ for $\theta\in C_+$. The exceptional divisor of $f_\theta\colon X\to \Spec R$ has five irreducible components denoted $E_1, \dots, E_5$.
 
 The key calculation of \emph{ibid}.\ associates one relation between the determinants of the tautological bundles for every such surface, namely
 \begin{align}
 T_{8^+} & \cong  T_{4^+}\otimes T_{4^+} & \text{for }E_1,\nonumber\\
 T_{6^-} & \cong  T_{2^-}\otimes T_{4^+} & \text{for }E_2, \nonumber\\
 T_{6^+} & \cong  T_{2^+}\otimes T_{4^+} & \text{for }E_3,\label{eqn:relations}\\
 T_{0^-} & \cong  T_{4^+}\otimes T_{8^-} & \text{for }E_4, \nonumber \\
 \det(T_5) & \cong  T_{4^+}\otimes \det(T_9) & \text{for }E_5.\nonumber
 \end{align}
 Corollary~\ref{cor:tilting} implies that $\{\det(T_{1}), T_{2^+}, T_{2^-}, T_{4^+}, T_{4^-}, T_{8^-}, \det(T_{9}), T_{10^+}, T_{10^-}\}$ provides a $\ZZ$-basis of $\NS(X)$, completing the calculation of Reid's recipe in the sense of Definition~\ref{def:RR1}. 
 
 To compute the matrices $K$ and $L$, list the basis elements $[P_i]\in \Theta$ for $i\neq \rho_{0^+}$ in the same order as in \eqref{eqn:irreps} and then use the relations \eqref{eqn:relations} to compute that 
 \[
\setlength{\arraycolsep}{2pt}
   \renewcommand{\arraystretch}{0.8}
L=\mbox{\footnotesize$\left(\begin{array}{cccccccccccrrrc}
0 & 1 & 0 & 0 & 0 & 0 & 0 & 0 & 0 & 0 & 0 & 0 & 0 & 0 \\
0 & 0 & 1 & 0 & 0 & 0 & 0 & 1 & 0 & 0 & 0 & 0 & 0 & 0 \\
0 & 0 & 0 & 1 & 0 & 0 & 0 & 0 & 1 & 0 & 0 & 0 & 0 & 0 \\
1 & 0 & 0 & 0 & 1 & 0 & 1 & 1 & 1 & 2 & 0 & 0 & 0 & 0 \\
0 & 0 & 0 & 0 & 0 & 1 & 0 & 0 & 0 & 0 & 0 & 0 & 0 & 0 \\
1 & 0 & 0 & 0 & 0 & 0 & 0 & 0 & 0 & 0 & 1 & 0 & 0 & 0 \\
0 & 0 & 0 & 0 & 0 & 0 & 1 & 0 & 0 & 0 & 0 & 1 & 0 & 0 \\
0 & 0 & 0 & 0 & 0 & 0 & 0 & 0 & 0 & 0 & 0 & 0 & 1 & 0 \\
0 & 0 & 0 & 0 & 0 & 0 & 0 & 0 & 0 & 0 & 0 & 0 & 0 & 1  
\end{array}\right).$}
\]
For example, the seventh column expresses $\det(T_5)$ as $T_{4^+}\otimes\det(T_9)$. Each of the five relations from \eqref{eqn:relations} determines a column of the matrix $K$, and again we choose to write its transpose
 \begin{equation}
    \label{eqn:Ktbento}
\setlength{\arraycolsep}{2pt}
   \renewcommand{\arraystretch}{0.8}
K^t=\mbox{\footnotesize$\left(\begin{array}{ccrrrrrrrrrrrr}
0 & 0 &  0 &  0 & -2  & 0 & 0 & 0 & 0 & 1 &  0 &  0 & 0 & 0 \\
0 & 0 &  0 & -1 & -1  & 0 & 0 & 0 & 1 & 0 &  0 &  0 & 0 & 0 \\
0 & 0 & -1 &  0 & -1  & 0 & 0 & 1 & 0 & 0 &  0 &  0 & 0 & 0 \\
1 & 0 &  0 &  0 & -1  & 0 & 0 & 0 & 0 & 0 & -1 &  0 & 0 & 0 \\
0 & 0 &  0 &  0 & -1  & 0 & 1 & 0 & 0 & 0 &  0 & -1 & 0 & 0    
\end{array}\right)$.}
\end{equation}
 Note that $K^t$ is sign-coherent. Inspecting the signs in each column leads to a trichotomy:
  
 \smallskip
 
 \indent \textsc{Case $(+)$}: The columns of $K^t$ containing one positive entry are indexed by the representations $\rho_{8^+}, \rho_{6^-}, \rho_{6^+}, \rho_{0^-}, \rho_{5}$ `marking' the surfaces $E_1, \dots, E_5$ respectively in \cite[Figure~6]{NolladeCelis21}. Assuming that Property~\ref{property} holds, then $k(i)=0$ for these $i$, and we deduce from Proposition~\ref{prop:GC} that
 \[
 \supp(\Psi(S_{i})) = E_j
 \]
 in each case. Based on the toric case, it is natural to anticipate that $\Psi(S_i) \cong \det(T_i)^{-1}\vert_{E_j}$ here. 
  
  Each $\rho_i$ marks the surface $E_j$ in \cite[Figure~6]{NolladeCelis21} precisely because certain $\rho_i$-semi-invariant functions lie in the socle of the $\Gamma$-clusters parametrised by these surfaces. We know $\theta_i:=\theta(\rho_i)>0$ for $\theta\in C_+$ and every such vertex $\rho_i$. The definition of $\theta$-semi-stability implies that these $\Gamma$-clusters become strictly semi-stable once we deform $\theta$ until $\theta_i=0$, so the hyperplane $\{\theta\in \Theta \mid \theta_i=0\}$ contains a wall of $C_+$ for each $i\in \{\rho_{8^+}, \rho_{6^-}, \rho_{6^+}, \rho_{0^-}, \rho_{5}\}$. This is consistent with Conjecture~\ref{conj:CLintro}\two;
 
 \smallskip
 
 \indent \textsc{Case $(0)$}: The four columns of $K^t$ that are the zero vector indicate that $\Psi(S_i)$ is supported on curves for the corresponding vertices $i\in \{\rho_1, \rho_{4^-}, \rho_{10^+}, \rho_{10^-}\}$; in each case, analysing $L^t$ shows that $\supp(\Psi(S_i))$ is the unique curve on which $\det(T_i)$ has degree one.
 
 We claim that for every such $i$, the hyperplane $\{\theta\in \Theta \mid \theta_i=0\}$ contains a wall of $C_+$. Indeed, analysing \cite[Table~7]{NolladeCelis21} indicates that, for example, $\rho_1$-semi-invariant functions lie in the socle of the $\Gamma$-cluster parametrised by the origins in charts denoted $U_4$ and $U_{14}$ (and for every cluster along the $\mathbb{P}^1$ joining these two points). As in \textsc{Case (+)} above, the clusters parametrised by this curve become strictly semi-stable for $\theta$ when $\theta_i=0$ and $\theta_j>0$ for $j\neq i$, proving the claim.  Therefore, Conjecture~\ref{conj:CLintro}\two\ holds if $k(i)=0$ for each $i\in \{\rho_1, \rho_{4^-}, \rho_{10^+}, \rho_{10^-}\}$, and it is natural to expect this in light of the toric case;

 \smallskip
 
 \indent \textsc{Case $(-)$}:  The columns of $K^t$ containing at least one negative entry are those indexed by $\rho_{4^+}, \rho_{2^-}, \rho_{2^+}, \rho_{8^-}, \rho_{9}$. If we assume Property~\ref{property}, then we deduce that 
     \begin{eqnarray*}
E_1\cup E_2\cup E_3\cup E_4\cup E_5 & = & \supp\big(\Psi(S_{4^+})\big),\\
E_2 & = & \supp\big(\Psi(S_{2^-})\big),\\
E_3 & = & \supp\big(\Psi(S_{2^+})\big),\\
E_4 & = & \supp\big(\Psi(S_{8^-})\big),\\
E_5 & = & \supp\big(\Psi(S_{9})\big).
\end{eqnarray*}
 Notice that the indices on the right-hand side are precisely those appearing on the right-hand sides of the relations \eqref{eqn:relations}, while the surface $E_j$ appears in the support of $\Psi(S_i)$ precisely when $\det(T_i)$ appears in the relation from  \eqref{eqn:relations} indexed by $E_j$. None of these five representations appear in the `Socle' column of \cite[Table~7]{NolladeCelis21}, so no $\Gamma$-cluster becomes strictly semistable if we deform $\theta$ from $C_+$ into any of the hyperplanes $\{\theta\in \Theta \mid \theta_i=0\}$. Thus, these hyperplanes do not contain any GIT walls of $C_+$; this too is consistent with Conjecture~\ref{conj:CLintro}\two\ since $k(i)=-1$.
  \end{example}

 \begin{example}[\textbf{Trihedral case}]
 \label{exa:trihedralNdC}
 A very similar analysis can be performed for the trihedral subgroup $\Gamma\subset \SL(3,\kk)$ described in \cite[Section~4]{NolladeCelis21}. In this case,  the supporting hyperplanes of $C_+$ of the form $\{\theta\in \Theta \mid \theta_i=0\}$ are those indexed by the irreducible representations denoted $\rho_{0'}, \rho_{0''}, V_1, V_7$, see \cite[Table~10]{NolladeCelis21}. Again, these correspond to the columns of $K^t$ that contain no negative entries which is consistent with Conjecture~\ref{conj:CLintro}\two. We leave the details to the reader.
 \end{example}
  
\bibliographystyle{plain}

 \end{document}